\documentclass[preprint]{imsart}

\usepackage{amsthm,amsmath,natbib}
\RequirePackage{natbib}
\usepackage{amsfonts}
\usepackage{esint,amssymb}
\usepackage{mathrsfs}
\usepackage{verbatim}    

\usepackage{bm}
\usepackage{bbm}
\usepackage{graphicx}
\usepackage{float}
\usepackage{enumerate}
\usepackage{stmaryrd}
\usepackage{mathtools}
\usepackage{blindtext}
\usepackage{appendix}
\usepackage{color}	

\usepackage{enumitem} 
\setlist[itemize]{noitemsep} 
\usepackage{enumitem}  

\usepackage{leftidx}   

\usepackage{soul}  

 \newcommand{\rom}[1]{
  \textup{\lowercase\expandafter{\romannumeral#1}} }
  
\newcommand{\Rom}[1]{
  \textup{\uppercase\expandafter{\romannumeral#1}} }

\startlocaldefs
\numberwithin{equation}{section}
\theoremstyle{plain}

\newtheorem{theorem}{Theorem}[section]

\newtheorem{lemma}{Lemma}[section]
\newtheorem{proposition}{Proposition}[section]
\newtheorem{remark}{Remark}[section]
\newtheorem{example}{Example}[section]
\newtheorem{assumption}{Assumption}[section]

\endlocaldefs

\newcommand{\reals}{{\mathbb R}}
\newcommand{\bbr}{\reals}
\newcommand{\bbz}{{\mathbb Z}}
\newcommand{\PP}{\mathbb{P}}

\newcommand{\one}{{\bf 1}}
\newcommand{\vep}{\varepsilon}
\newcommand{\ProbSpace}{(\Omega, {\mathcal F}, \PP)}
\newcommand{\eid}{\stackrel{d}{=}}

\begin{document}

\begin{frontmatter}

\title{Extremal clustering under moderate long range dependence and moderately
  heavy tails}
\runtitle{Extremal clustering}
\thankstext{T1}{ This research was partially supported by the ARO grant 
 W911NF-18 -10318 at Cornell  University.}

\begin{aug}
\author{\fnms{Zaoli} \snm{Chen}  \ead[label=e1]{zc288@cornell.edu}} 

 \author{\fnms{Gennady} \snm{Samorodnitsky} 
 \corref{}\ead[label=e2] {gs18@cornell.edu}}        

\runauthor{Zaoli Chen and Gennady Samorodnitsky}

\affiliation{Cornell University}

\address{Department of Mathematics \\
  Cornell University \\
  \printead{e1}}

\address{School of Operations Research and Information Engineering \\
  Cornell University \\
  \printead{e2}}

 \end{aug}

\begin{abstract}
We study clustering of the extremes in a stationary sequence with
subexponential tails in the maximum domain of attraction of the Gumbel
We obtain functional limit theorems in the space of random sup-measures 
and in the space $D(0,\infty)$. The limits
have the Gumbel distribution if the memory is only moderately
long. However, as our results demonstrate rather strikingly, the
``heuristic of a 
single big jump'' could fail  even in a moderately long range
dependence setting. As 
the tails become lighter,  the extremal behavior of a stationary
process may  depend on multiple large values of the driving noise. 
\end{abstract}



\end{frontmatter}



\section{Introduction}   \label{sec;intro}
  
This paper is about a very unusual clustering of extreme values that can
occur in certain types of stationary stochastic processes with long
range dependence. It is useful to recall the basic definitions of the
classical extreme value theory. A distribution $H$ on $\bbr$ is in a maximum domain of
attraction if there is a positive sequence $(a_n)$ and a real sequence
$(b_n)$ such that the law of $(M_n^{(0)}-b_n)/a_n$ converges weakly as
$n\to\infty$ to a nondegenerate distribution $G$. Here
$M_n^{(0)}=\max(Y_1,\ldots, Y_n)$ is the largest  value among $n$ i.i.d. random
variables $Y_1,Y_2,\ldots$ with the common distribution $H$. The distribution $G$ is
then, automatically, of the form $G(x)=G_\gamma(Ax+B), \, x\in\bbr$
for some $A>0, B\in\bbr$, and some $\gamma\in\bbr$. The ``standard''
distributions $G_\gamma$ are the Fr\'echet $G_\gamma(x)=\exp\{-x^{-1/\gamma}\}, \,
x\geq 0$ if $\gamma>0$, the Gumbel $G_0(x)=\exp\{ -e^{-x}\}, \, x\in\bbr$,
and the Weibull $G_\gamma(x) = \exp\{ -(-x)^{-1/\gamma}\}, \, x\leq 0$ if
$\gamma<0$. See e.g. \cite{dehaan:ferreira:2006}.

The extreme values of an i.i.d. sequence, obviously, do not
cluster. If, on the other hand, $X_1,X_2,\ldots$ is a stationary
sequence with a common marginal distribution $H$, its extreme values
may exhibit a clustering phenomenon. This is a well studied topic in
the extreme value theory, where a numerical measure of clustering,
{\it the extremal index}, goes back to \cite{leadbetter:1983}. Let
$M_n=\max(X_1,\ldots, X_n)$, $n\geq 1$. As above, we denote by
$M_n^{(0)}$ the largest of the first $n$ {\it i.i.d.} observations
$Y_1,Y_2,\ldots$ with the same marginal distribution $H$. Then the
stationary sequence $X_1,X_2,\ldots$ has extremal index $\theta$ if
for some nondegenerate distribution $G$ we have both
\begin{equation} \label{e:extremal.index}
(M_n^{(0)}-b_n)/a_n \Rightarrow G \ \ \text{and} \ \ (M_n-b_n)/a_n
\Rightarrow G^\theta
\end{equation}
as $n\to\infty$; see e.g. \cite{dehaan:ferreira:2006}. An extremal
index, if it exists, is in the range $0<\theta\leq 1$. Note that the
first statement in \eqref{e:extremal.index} can be rewritten in the
form
$$
(M_{[n\theta]}^{(0)}-b_n)/a_n \Rightarrow G ^\theta\,,
$$
which, in conjunction with the second statement in
\eqref{e:extremal.index}, says that the largest among the first $n$
observations  from the stationary sequence is
``similar'' to the largest of the first $[n\theta]$ observations from
the corresponding i.i.d. sequence. In fact, in most cases a stationary
sequence satisfying \eqref{e:extremal.index} also satisfies
\begin{equation} \label{e:Mn.theta}
  (M_n-b_{[n\theta]})/a_{[n\theta]}\Rightarrow G\,,
\end{equation}
which emphasizes the similarity between $M_n$ and $M_{[n\theta]}^{(0)}$
even more. It is also a part of the folklore in extreme value theory
that the extremal index can be interpreted as the reciprocal of ``the
expected extremal cluster size'', though we will introduce neither the exact
definition of this object  nor the conditions under which this
interpretation is valid; see
e.g. \cite{embrechts:kluppelberg:mikosch:1997}. If we denote such
expected extremal cluster size by $\kappa$, then an alternative
expression of \eqref{e:Mn.theta} is 
\begin{equation} \label{e:Mn.cluster}
  (M_n-b_{[n/\kappa]})/a_{[n/\kappa]}\Rightarrow G\,.
\end{equation}

A special situation occurs when the stationary sequence
$X_1,X_2,\ldots$ exhibits long range dependence with respect to its
extremes (see \cite{samorodnitsky:2016}). In this case
\eqref{e:extremal.index} may hold with $\theta=0$ which, of course,
only says that the centering and scaling of $M_n$ should be changed to
obtain a nondegenerate limit. Intuitively, the extremes cluster so 
much that the extremal clusters become unbounded. and one expects that
\eqref{e:Mn.cluster} should be replaced by
\begin{equation} \label{e:Mn.clusterL}
 (M_n-b_{m_n})/a_{m_n}\Rightarrow G\,,
\end{equation}
where $m_n=[n/\kappa_n]$, and 
now $\kappa_n$ is ``the expected extremal cluster size'' among
the first $n$ observations. This would allow $\kappa_n\to\infty$ and,
therefore, $m_n=o(n)$ as $n\to\infty$. Hence, a change in the order of
magnitude of scaling and/or centering for the partial maximum $M_n$. 

In fact, it has turned out that \eqref{e:Mn.clusterL} holds for
certain stationary infinitely processes with regularly varying
tails. In this case the marginal distributions are in the Fr\'echet
maximum domain of attraction ($\gamma>0$), which involves no centering
($b_n\equiv 0$); see \cite{samorodnitsky:2004a} and
\cite{lacaux:samorodnitsky:2016}. In this setup the sequence $(m_n)$
in \eqref{e:Mn.clusterL} was not obtained via the relation
$m_n=[n/\kappa_n]$ but, rather, turned out to be a direct ingredient 
in the memory in the system. It is important to mention that in these
cases an important distinction has appeared between ``moderate'' long
range dependence and ``extreme'' long range dependence. In the former
case the weak limit in \eqref{e:Mn.clusterL} is, up to shifting and
scaling, the standard Fr\'echet distribution, while in the latter case
the limit is not one of the classical extreme value distributions; see
\cite{samorodnitsky:wang:2019}. In the former case an extreme of the
process  due to a single large value of the underlying noise,
consistent with the ``heuristic of a single big jump'' for extreme
events and large deviations of heavy tailed systems (e.g. 
\cite{rhee:blanchet:zwart:2019}). On the other hand, in the latter
case this heuristic fails.

Our goal in this paper is to understand how the extremes of a
long memory stationary process cluster when the marginal tails 
are still heavy (so that the ``heuristic of a
single big jump'' is still often the first guidance one has), but less
heavy than the regularly varying tails considered earlier. The natural
class of such marginal distributions is the class of {\it
  subexponential distributions}. Recall that a  distribution $H$ is
subexponential if
\begin{equation} \label{e:subexp.def}
\lim_{x\to \infty}  \frac{\overline{H \ast H} (x)}{\overline{H}(x)}=2\,,
\end{equation}
where $\overline{H}(x)=1-H(x)$ (\cite{chistyakov:1964}). Distributions
with a regularly varying right tail are, of course, subexponential,
but we are interested in the subexponential distributions whose tails
are lighter than any regularly varying tail. Specifically, we are
interested in the subexponential distributions in the maximum domain
of attraction of the Gumbel distribution $G_0$. We will give the exact
assumptions on the marginal tails in the sequel.

A surprising conclusion of our results that, while
\eqref{e:Mn.clusterL} does continue to hold in the class of long
memory stationary proceses with 
certain subexponential
distributions in the  maximum domain 
of attraction of the Gumbel distribution $G_0$ as the marginal
distributions, \eqref{e:Mn.clusterL} breaks down for such
distributions once their tails become light enough. This may happen
even when the memory is only ``moderate long memory'' (to be defined
precisely in the sequel.) The moderate long memory case is the only
one we consider in this paper. It turns out that when the tails become
light enough, the centering in the left hand side of ,
\eqref{e:Mn.clusterL} acquires another term, of a smaller order than
$b_{m_n}$ (but of a larger order than $a_{m_n}$.) This term arises
because the ``single big jump'' heuristic breaks down once again. That
is, that heuristic may break down not only when the memory is too
long, but also when the tails are too light (while remaining
subexponential, hence heavy!)

The paper is organized as follows. In Section \ref{sec:prelim} we
review some essential facts on subexponential distributions in the
Gumbel maximum domain of attraction, random closed sets,
null-recurrent Markov chains underlying the infinitely divisible
dynamics in our model  and random sup-measures, and introduce the
limiting random sup-measure later appearing in the main 
result. In Section \ref{sec;sidp}, we introduce the stationary
infinitely divisible processes we are considering and list the
assumptions we are imposing.  The main results, the extremal limit
theorems in the space of random sup-measures and in the space of
c\`{a}dl\`{a}g functions are stated and proved in Section \ref{sec;elt}. This
sections also contains two natural examples. 
The appendices \ref{app;+rwrv} and  \ref{app;supp} contain several
arguments and verifications needed in the earlier parts of the paper. 

We will use the following standard notation throughout the paper. 
Let $\{a_n\}_{n \in \mathbb{N}}$ and $\{b_n\}_{n \in \mathbb{N}}$ be
two positive sequences. We describe the asymptotic relation between
them by writing: 
\begin{enumerate} [label=(\alph*)]
    \item  $a_n \sim b_n$ if 
      $\lim_{n \to \infty}a_n / b_n =1$,
      
      \item $a_n \gg b_n$ if $\lim_{n \to \infty} a_n  / b_n = \infty$, or equivalently
      $b_n \ll a_n$, 

     \item  $a_n \lesssim b_n $ if there exists $C>0$ such that $a_n
       \leqslant C b_n$  for large enough $n$, and analogously with
       $a_n \gtrsim b_n $, 

    \item $a_n \asymp b_n$ if both $a_n \lesssim b_n $ and $a_n
      \gtrsim b_n $.
    \end{enumerate}

    If $\{A_n\}_{n \in \mathbb{N}}$ and $\{B_n\}_{n \in \mathbb{N}}$
    are two sequences of positive random variables, we will write
    
   \begin{enumerate} [label=(\alph*)]
    \item $A_n = o_P(B_n)$ if $A_n / B_n \to 0$ in probability, 
    \item $A_n\lesssim_P B_n$ if  $(A_n/B_n)$ is tight, and
      analogously with $A_n \gtrsim_P B_n$. 
\end{enumerate}


\section{Preliminaries} \label{sec:prelim}

This section is of the background nature. It collects a number of
mostly well-known notions and results needed in this paper. 

\subsection{Subexponential distributions in the Gumbel maximum domain
  of  attraction} \label{subsec:gumbel}

Most of the material quoted in this section is in \cite{resnick:1987};
see also \cite{goldie:resnick:1988}. Subexonentiality requires the
distribution to have a support which is unbounded on the right, so we
only consider such distributions. 
A  distribution $H$ is in the maximum domain
  of  attraction of the Gumbel distribution if and only if $G= ( 1 /
  (1-H))^\leftarrow $ is $\Pi$-varying, and for a non-decreasing
  function $J$ the generalized inverse of $J$ is defined as  
\begin{equation} \label{eq;inverse1}
    J^{\leftarrow} (y) = \inf \{ s : J(s) \geqslant y  \}\,.
  \end{equation}
Furthermore, a non-negative, non-decreasing function $V$ is said to be
$\Pi$-varying if there exist functions $a(t)>0, b(t) \in \mathbb{R}$
such that for $x>0$ 
\begin{equation} \label{e:pi}
    \lim_{t \to \infty} \frac{V(tx) - b(t)}{a(t)} = \log x . 
\end{equation} 
Alternatively, $H$ is in the maximum domain
  of  attraction of the Gumbel distribution if and only if there exist
  $x_0 \in \mathbb{R}$ and $c(x)\to c>0$ as $x\to\infty$ such
  that for $x_0< x <  \infty$ 
\begin{equation} \label{eq;VM1}
    \overline{H}(x) = c(x) \exp \left \{
    - \int_{x_0} ^ x \frac{1}{h(u)} du 
    \right \}
\end{equation}
where $h$ (the
so-called auxiliary function) is an absolutely continuous
positive function  on $(x_0,\infty)$ with density $h^\prime$ satisfying
$\lim_{u \to \infty} h^\prime (u) =0$. The function $h$ must satisfy
$h(x)=o(x)$ as $x\to\infty$, and subexponentiality of $H$
requires also $\lim_{u \to \infty} h(u) =\infty$. For a distribution
$H$ satisfying \eqref{eq;VM1}, the centering and scaling required for
the convergence $(M_n^{(0)}-b_n)/a_n\to G_0$ can be chosen as 
$$
        b_n = \left(  \frac{1}{1-H} \right ) ^ \leftarrow  (n) , \ 
         a_n = h ( b_n )\,.
         $$
We will often use the following fact: if one
replaces the function $c(\cdot)$ in \eqref{eq;VM1} by an asymptotically
equivalent function, and denotes the new normalizing sequences by
$(\tilde a_n)$ and $(\tilde b_n)$, then
\begin{equation} \label{e:small.diff}
  \lim_{n\to\infty} \frac{b_n-\tilde b_n}{a_n}=0, \ 
  \lim_{n \to \infty} \frac{ \tilde a_n } {a_n } = 1 \, .
  \end{equation}

  \subsection{Random closed sets} \label{subsec:RCS}
  
We use the notation  $\mathcal{G}, \mathcal{F}$ and $\mathcal{K}$  for
the  families of open, closed and compact sets of $[0,1]$ or
$[0,\infty)$ (depending on the context), respectively. Details on most of the
material in this section can be found in \cite{molchanov:2017}.

The Fell topology $\mathcal{B}(\mathcal{F})$ on 
$\mathcal{F}$ is the topology generated by the subbasis
\begin{align*}
 &\mathcal{F}_G =   \left\{ F \in \mathcal{F} : F \cap G \neq \emptyset  \right \}, \quad G \in \mathcal{G},  \\
 &\mathcal{F}^K =  \left\{ F \in \mathcal{F} : F \cap K  =  \emptyset  \right \} , \quad K \in \mathcal{K}. 
\end{align*}
The Fell topology is metrizable and compact. 
A random closed set is a measurable mapping from a probability space
to $\mathcal{B}(\mathcal{F})$.  

For $\beta\in (0,1)$, let $\{ L_\beta (t) \}_{t \geqslant 0}$ be the
standard $\beta$-stable subordinator, i.e. the increasing L\'{e}vy
process with  the Laplace transform $\mathbb{E}e^{-\theta L_\beta(t)} = e
^{- t \theta ^  \beta}, \theta \geqslant 0$. We define the $\beta$-stable regenerative set $Z$ to be the closure of the range of a $\beta$-subordinator, i.e.,
\begin{equation} \label{eq;betaSRS}
    Z =\overline{\{ L_\beta(t): t \geqslant 0 \}}.
\end{equation}
It is a random closed subset of $[0,\infty)$. Much of the discussion
in this paper revolves around a sequence of i.i.d. random closed
subsets of $[0,1]$ defined as follows. 

Let $\{V_j\}_{j\geqslant 1}$ be a family of i.i.d. random variables on
$[0,1]$ with
\begin{equation} \label{eq;distV}
    \mathbb{P}( V_1 \leqslant x ) = x ^ {1-\beta}, \quad x \in
    [0,1]\,,
\end{equation}
independent of an i.i.d. sequence $\{Z_j \}_{j \geqslant 1}$ of
$\beta$-stable regenerative sets. We define 
\begin{equation}  \label{eq;tctSRS}
    \overline{R_j} = ( V_j + Z_j ) \cap [0,1].
\end{equation}
That is, each $ \overline{R_j}$ is the restriction of a shifted stable
regenerative set to  $[0,1]$. It is, clearly, non-empty.

\subsection{Null-recurrent Markov chains} \label{subsec:nullrec}

We describe now some ergodic theoretic notions associated with certain
null recurrent Markov chains.  Our main reference here is
\cite{aaronson:1997}. 
Let $\{ Y_t \}_{t\in \mathbb{Z}}$ be an irreducible, aperiodic, and
null recurrent Markov chain on  $\mathbb{Z}$, and denote by $(\pi_i)_{i \in
  \mathbb{Z}}$ its unique  invariant measure satisfying $\pi_0 =1$. 
If $(E, \mathcal{E})$ is the 
``path space'' $(\mathbb{Z} ^ \mathbb{Z} , \mathcal{B} (\mathbb{Z} ^
\mathbb{Z}) )$, we can define an infinite $\sigma$-finite measure on
$(E,\mathcal{E})$ by 
\begin{equation} \label{eq;mu}
    \mu (\cdot) : = \sum_{i\in \mathbb{Z}} \pi_i P_i (\cdot)\,,
  \end{equation}
where $P_i$  is the probability law of $\{Y_t\}_{t\in \mathbb{Z}}$ on $(E,\mathcal{E})$ 
given $Y_0=i$. The  left shift operator on $E$ by $\theta$ defined by 
\begin{equation}  \label{eq;theta}
    \theta: (\ldots,y_0, y_1,y_2,\ldots) \mapsto 
(\ldots, y_1, y_2, y_3,\ldots) 
\end{equation}
is a measure preserving, conservative and ergodic operator on 
 $(E,\mathcal{E},\mu)$. See \cite{harris:robbins:1953}.  
For  $n \in \mathbb{Z}$ let 
\begin{equation} \label{eq;An}
    A_n : = \{  y \in E : y_n = 0   \}.
\end{equation} 
The {\it wandering rate sequence} $\{w_n\}_{n\in \mathbb{N}}$ is then
defined  as
\begin{equation}  \label{eq;wdr}
    w_n : = \mu \left ( \bigcup_{k=0} ^ {n} A_k \right ), \quad n \in \mathbb{N}.
\end{equation}
 
Define the first visit time to state $0$ by 
\begin{equation} \label{eq;varphi}
\varphi (y) : = \inf \{ t\geqslant 1 : y_t = 0 \}, y \in E. 
\end{equation}
The Markov chains we consider satisfy the following assumption. 
\begin{assumption}  \label{ass;YRV0}
There exists $\beta\in (0,1)$ and a slowly varying function $L$ such that 
\begin{align}  
    &\overline{F} (n) : = P_0 [\varphi > n  ] = n ^{-\beta}
      L(n) \in \text{RV}_{-\beta}\,.
    \label{eq;varphi1}
\end{align}
Furthermore,
\begin{align}
 & \sup_{n \geqslant 0} \frac{n \mathbb{P}_0 (\varphi = n)
   }{\overline{F}(n)} < \infty\,. 
  \label{eq;varphi2} 
\end{align}
\end{assumption}
 
\begin{remark} \label{rk:return}
{\rm 
  Under the Assumption \ref{ass;YRV0}, as $n\to\infty$,
 \begin{align}
w_n & \sim  \sum_{k=1}^n \mu (  \varphi =k   ) \notag \\
    & =  \sum_{k=1} ^ n \overline{F} (k-1) 
    \sim \frac{ n^{1-\beta} L(n)}{1-\beta} \in \text{RV}_{1-\beta}\,.\label{e:wn}
\end{align}
See \cite{resnick:samorodnitsky:xue:2000} Lemma 3.3.
 }
\end{remark}

The times a sequence $y\in E$ visits state 0 under certain conditional
versions of the measure $\mu$ are of crucial importance for
us. Specifically, for  each $n\in \mathbb{N}$, let
\begin{equation}  \label{eq;mun}
    \mu_n (B) : = \frac{ \mu\left(B \cap  \bigcup_{k=0}^n A_k \right )
    } { w_n }, \  B \in \mathcal{E}. 
\end{equation}
Let $\{ Y^{(j,n)}  \}_{j\in \mathbb{N}}$ be a family of  i.i.d. random
elements in $E $ with law $\mu_n$. For each $j$ we set 
\begin{equation}  \label{eq;Ijn}
    I_{j,n} : = \left \{ 0 \leqslant t \leqslant n : Y^{(j,n)}_t = 0
    \right  \}\,.
\end{equation}
We further define
\begin{align} 
    & \widehat{I}_{1,n} : = I_{1,n}  \label{eq;hatI1} , \\
    & \widehat{I}_{j,n} : = I_{j,n} \cap \bigcap_{i=1}^{j-1}
      I_{i,n}^c,  \quad j \geqslant 2  \label{eq;hatIj}.
\end{align}

The facts mentioned below are in \cite{samorodnitsky:wang:2019}. First, by
Theorem 5.4 {\it ibid.}, 
 \begin{equation} \label{eq;mtg3.1}
     \frac{1}{n}I_{j,n} \Rightarrow \overline{R_j}, \ j=1,2,\ldots
   \end{equation}
   weakly in the space of random closed subsets of $[0,1]$, 
   where $ \overline{R_j}$ is defined in (\ref{eq;tctSRS}). In
   particular,  
    \begin{equation}  \label{eq;mtg3.2}
        \lim_{n \to \infty} \mathbb{P}\left(
        I_{j,n} \cap n G \neq \emptyset
        \right)  = \mathbb{P}\left(
        \overline{R_j} \cap  G \neq \emptyset
        \right)  > 0  \; \text{for any } G \in \mathcal{G}([0,1]).
    \end{equation} 
 If, in addition, $0<\beta<1/2$, then 
for any two distinct $j_1, j_2 \in \mathbb{N}$
   \begin{equation}  \label{eq;mtg3.3}
     \lim_{n\to \infty}  \mathbb{P} \left(
     I_{j_1,n} \cap I_{j_2,n} \not= \emptyset
     \right ) = 0 . 
   \end{equation}
 Therefore,  for any $m \in \mathbb{N}$, 
    \begin{equation} \label{eq;mtg3.4}
      \lim_{n\to \infty} \mathbb{P} \left(
      I_{j,n}= \widehat{I}_{j,n} , \, j=1,\ldots,m.
      \right ) = 1 .
    \end{equation} 

In the sequel we will need estimates of how quickly the intersection
probability in \eqref{eq;mtg3.3} and certain related probabilities
converge to zero. For an open interval $T \subset [0,1]$ we define
\begin{align}
    &    p_{n,T} : = \mathbb{P}\left( I_{1,n} \cap I_{2,n}  \cap nT \neq \emptyset \right)  \label{eq;anneal} ,\\
    &  \overline{p }_{n,T}:  
    = \mathbb{P}\left(  I_{1,n}  \cap I_{2,n}  \cap nT \neq \emptyset
      \big| Y^{(1,n)}  \right )\,, 
    \label{eq;quench}
\end{align}
with the latter probability being random. Clearly,  $ p_{n,T} =
\mathbb{E} \overline{p}_{n,T} $. The following theorem may be of
independent interest. It is proved in Appendix \ref{app;+rwrv}. 
 
\begin{theorem} \label{thm;istP}
Under Assumption \ref{ass;YRV0} with $0<\beta<1/2$, for any open
interval $T$,  
\begin{enumerate}  [label=(\roman*)]
    \item  
    \begin{equation} \label{eq;anl1}
        p_{n,T} \asymp  \frac{  n^{\beta}  }{w_n L(n)}\,. 
    \end{equation}
    \item 
    For any $C>0$, there exists  $c>0$ such that for every $n\geq 1$ 
    \begin{equation} \label{eq;qch1}
        \mathbb{P}\left( 
        \overline{p}_{n,T} \geqslant \frac{ c n^{\beta} \log n }{w_n L(n)} 
        \right)  \leqslant n^{-C}\,.
    \end{equation}
    \item 
    For any $\gamma > (1-2\beta)^{-1}$ and $\epsilon>0$, there exists
    $c>0$ such that 
    \begin{equation} \label{eq;qch2.1}
    \liminf_{n\to \infty}   \mathbb{P} \left( \overline{p}_{n,T} \geqslant  \frac{c n^\beta  }{w_n  L(n)}  \cdot  \frac{ L\left( (\log n)^\gamma  \right )}{(\log n)^{\gamma \beta} }
       \, \Big \vert \,  I_{1,n} \cap \, nT  \neq \emptyset
       \right) \geqslant 1 - \epsilon . 
    \end{equation}
 \end{enumerate}

\end{theorem}

It follows immediately from part (iii) of the theorem that
  \begin{equation} \label{eq;qch2.2}
        \overline{p}_{n,[0,1]} \gtrsim_P
        \frac{n^\beta  }{w_n  L(n)}  \cdot  \frac{ L\left( (\log
            n)^\gamma  \right )}{(\log  n)^{\gamma \beta} }\,.
     \end{equation}       
for any $\gamma > (1-2\beta)^{-1}$.

\subsection{Random Sup-Measures} \label{sec:rsm}

We continue to use the notation of Subsection \ref{subsec:RCS}. Our
main reference is \cite{obrien:torfs:vervaat:1990}; note that our
sup-measures take values in $\overline{\mathbb{R}}=[-\infty,
\infty]$. 

A  sup-measure is a mapping $m:\mathcal{G} \to \overline{\mathbb{R}}$ such that $m(\emptyset) =-\infty$ and 
$ m (  \cup_\alpha G_\alpha ) = \vee_\alpha m (G_\alpha) $ for an
arbitrary collection $( G_\alpha )$ of open sets. 
The sup-derivative $d^\vee m$ of $m$ is defined by
$$
        d^ \vee m (t) = \bigwedge_{t \in G} m (G) ,
$$
it is an upper semicontinuous $\overline{\mathbb{R}}$-valued 
function of $t$. Given an $\overline{\mathbb{R}}$-valued 
function $f$,  the sup-integral of $f$ defined by
$$
        i^ \vee f (G) = \bigvee_{t \in G } f (t), \quad  G \in \mathcal{G};
$$
it is automatically a sup-measure. It is always true that  $m = i^
\vee d ^ \vee m$, and we can extend the 
 domain of a sup-measure to all Borel sets via  
    \begin{equation} \label{eq;SMext}
        m(B) = \bigvee_{t \in B}  d^\vee m (t), \quad B \ \ \text{Borel.}
      \end{equation}

   On the collection $\text{SM}$ of all sup-measures one defines
the sup-vague topology, in which a sequence of sup-measures
$\{m_n\}_{n \geqslant 1}$ converges to a sup-measure $m$
if and only if   
    \begin{align*}
        & \limsup_{n \to \infty}  m_n(K) \leqslant m(K), \quad 
        \text{for each }K \in \mathcal{K}, \\
        & \liminf_{n \to \infty}  m_n(G) \geqslant m(G), \quad 
        \text{for each }G \in \mathcal{G}.
    \end{align*}
The space $\text{SM}$ equipped with sup-vague topology is compact and metrizable.

A random sup-measure $M$ is a measurable mapping from a probability
space to SM.  For a random sup-measure $M$, let $\mathscr{I}(M)$ be
the collection of continuity intervals of $M$, defined by 
$$
\mathscr{I}(M) = \{ I \ \text{an open interval}:\ M(I) = M(\text{clos } I ) \text{ a.s.} \}.
$$
If $\{M_n\}_{n \geqslant 1}$ and $M$ are random sup-measures, then
$M_n \Rightarrow M$ if and only if 
\begin{equation} \label{eq;rsmWC1}
    (M_n (I_1),\ldots,M_n (I_m) ) \Rightarrow 
    (M (I_1),\ldots,M (I_m) )
\end{equation}
for arbitrary disjoint intervals $I_1,\ldots,I_m \in \mathscr{I}(M)$.

We now define a family of random sup-measures that will arise
naturally in the sequel. Let $\beta \in (0,1/2)$ and consider a
Poisson point process on  
$\mathbb{R} \times  \mathbb{R}_+ \times \mathcal{F}(\mathbb{R}_+) $ with mean measure 
$$
e^{-u} d u  (1-\beta)v^{-\beta} d v \, d P_\beta\,,
$$
where $P_\beta$ is the law of the $\beta$-stable regenerative set in
\eqref{eq;betaSRS}.  Let $(U_j, V_j^*, Z_j)_{j \in \mathbb{N} }$ be a
measurable enumeration of points of this Poisson point process, and
denote 
\begin{equation}  \label{eq;sftSRS}
    R_j = V_j^* + Z_j, \quad j \in \mathbb{N}.
\end{equation}
Since $\beta\in (0,1/2)$, we have 
\begin{equation} \label{eq;SRScap}
    \mathbb{P}(R_1 \cap R_2 = \emptyset) =1\,;
\end{equation}
see Lemma 3.1 in \cite{samorodnitsky:wang:2019}. It follows
immediately that, on an event of probability 1, the function
$\eta:\mathbb{R}_+ \to \overline{\mathbb{R}}$ defined by 
$$
    \eta (t)  = \bigvee_{j=1} ^ \infty U_j 1_{\{   t \in R_j  \} } 
$$
is an upper semicontinuous function. Hence, it is the sup-derivative
of the random sup-measure
\begin{equation} \label{eq;limRSM1}
    \mathcal{M}(B) = \bigvee_{j=1}^\infty U_j 1_{  \{B\cap R_j\not=\emptyset \} }\,.
  \end{equation}
This measure is stationary, i.e. 
    \begin{equation}\label{eq;limRSM1.1}
        \mathcal{M}( r + \cdot ) \stackrel{d}{=} \mathcal{M}(\cdot)
    \end{equation}
    for any $ r \geqslant 0$; see Proposition 4.3 in
    \cite{lacaux:samorodnitsky:2016}.  Moreover, we claim that $M$ is
    self-affine, i.e. 
    \begin{equation}\label{eq;limRSM1.2}
        \mathcal{M}(a \cdot ) \stackrel{d}{=} \mathcal{M}(\cdot)
        +  (1 - \beta) \log a 
    \end{equation}
   for any $a>0$.  Indeed, note that  
\begin{align*}
   \mathcal{M}(a B ) 
    & =  \bigvee_{j=1}^ \infty U_j 1_{ \{ B\cap (a^{-1}V_j^* +  a^{-1} Z_j) \not=\emptyset  \} } \\
    & \stackrel{d}{=} \bigvee_{j=1}^ \infty U_j 1_{ \{ B\cap (a^{-1}V_j^* +   Z_j) \not=\emptyset  \} }  \\
    & \stackrel{d}{=} 
    \bigvee_{j=1}^ \infty ( U_j + (1-\beta)\log a )   1_{ \{ B\cap (V_j^* +   Z_j) \not=\emptyset  \} } =   \mathcal{M}(B )  +  (1 - \beta) \log a\,;
\end{align*}
see e.g. Proposition 4.1(b) in \cite{samorodnitsky:2016} for the
first distributional equality, while the second one holds because 
both the points $(U_j, a^{-1}V_j^*, Z_j)_{j \in \mathbb{N} }$ and the points
$(U_j + (1-\beta)\log a, V_j^*, Z_j)_{j \in \mathbb{N} }$ form a  
Poisson point process with mean measure 
$$
a^{1-\beta}e^{-u} d u  (1-\beta)v^{-\beta} d v \, d P_\beta\,. 
$$

\medskip

Suppose that $\{X_t\}_{t \in \mathbb{Z} }$ is a stationary process. It
induces naturally 
a sequence of random sup-measures by setting for $n\in
\mathbb{N}$ 
\begin{equation} \label{e:sup.m.n}
M_n (B) = \max_{t \in n B} X_t , \quad B \in \mathcal{B}(\mathbb{R}_+)\,.
\end{equation} 
If $\mathcal{M}$ is the random sup measure in \eqref{eq;limRSM1}, then
the  weak convergence 
$$
\frac{M_n(\cdot) - b_n}{a_n} \Rightarrow \mathcal{M}(\cdot)
$$ 
 in the space of random sup-measures on $[0,1]$ for some $(a_n,b_n)$
 guarantees also this weak convergence in the space of random 
 sup-measures on $\bbr_+$. Furthermore, every open interval is a
 continuity interval for  $\mathcal{M}$, since stable regenerative
 sets do not hit fixed points. By (\ref{eq;rsmWC1})  
\begin{equation}  \label{eq;rsmWC2}
    \left( \frac{M_n(I_1) - b_n}{a_n}, \ldots, \frac{M_n(I_m) - b_n}{a_n}
\right )\Rightarrow ( \mathcal{M}(I_1), \ldots, 
 \mathcal{M}(I_m) )
\end{equation}
for arbitrary disjoint open intervals $I_1,\ldots, I_m$ in $[0,1]$ 
is necessary and sufficient for weak  convergence to $\mathcal{M}$.

The restriction of the sup measure $\mathcal{M}$ to subsets of $[0,1]$
has  representation somewhat more transparent than
\eqref{eq;limRSM1}.  Let $\{V_j\}_{j\geqslant 1}$ be a family of
i.i.d. random variables on $[0,1]$ with  the law \eqref{eq;distV}.
Let $\{Z_j \}_{j \geqslant 1}$ be a family of i.i.d. $\beta$-stable
regenerative sets in (\ref{eq;betaSRS}).   Finally, 
let $\{\Gamma_j\}_{j\geqslant 1}$ be the sequence of arrival times of
a unit rate Poisson processes on $(0,\infty)$. We assume that all
three sequences are independent. Then the points $(-\log \Gamma_j,
V_j, Z_j)_{j \in 
  \mathbb{N} }$ form a Poisson point process on $\mathbb{R} \times
[0,1] \times \mathcal{F}(\mathbb{R}_+) $ whose mean measure is the
mean measure of restriction of the Poisson point process $(U_j, V_j^*,
Z_j)_{j \in   \mathbb{N} }$ to $\mathbb{R} \times
[0,1] \times \mathcal{F}(\mathbb{R}_+) $. Therefore, if we define
i.i.d. random nonempty compact sets by \eqref{eq;tctSRS}, 
then the following representation in law holds: 
\begin{equation} \label{eq;limRSM2}
    \mathcal{M}(B) = \bigvee_{t\in B} -\log \Gamma_j 1_{\{  B\cap
      \overline{R_j} \neq \emptyset \} }, \ B \in
    \mathcal{B}([0,1])\,. 
\end{equation}


\section{A family of stationary subexponential infinitely divisible
  processes} 
\label{sec;sidp}

We now define a family of stationary infinitely divisible processes
for whom we will establish extremal limit theorems. Our processes will
be of the form 
\begin{equation} \label{eq;seqX}
    X_n = \int_E  f\circ\theta^n(x)\, M(dx), \ n\in \mathbb{Z}\,,
 \end{equation}
where $\theta$ is the left shift operator on $E=\bbz^\bbz$ in
\eqref{eq;theta} and $M$ is an infinitely divisible  random measure 
on $(E, \mathcal{E})$ with a constant local characteristic triple
$(\sigma^2, \nu, b)$ and control measure $\mu$ in \eqref{eq;mu},
associated with an invariant measure of an irreducible, aperiodic, and
null recurrent Markov chain on  $\mathbb{Z}$; see Chapter 3 in
\cite{samorodnitsky:2016} for details in infinitely divisible random
measures and integrals with respect to such measures. The function $f$
must satisfy certain integrability conditions; if it does, the process
$\{X_n\}_{n\in \mathbb{Z}}$ is automatically stationary, because the
left shift $\theta$ preserves 
the control measure $\mu$. In the sequel we will assume, for
simplicity, that $f$ is the indicator function
\begin{equation} \label{e:f.ind}
f(x) = \one(x_0=0) \ \ \text{for} \ \  x= (\ldots,x_0, x_1,x_2,\ldots)\,,
\end{equation} 
but the results of this paper will undoubtedly hold for a more general
class of functions $f$.  The indicator function $f$ in
\eqref{e:f.ind} automatically satisfies the integrability conditions
and, in this case, each $X_n$ is an infinitely divisible random
variable with a characteristic triple $(\sigma^2, \nu, b)$; see
Section 7 in \cite{sato:2013}. 
The key assumption we will impose in the sequel is that the
distribution $(\nu(1,\infty))^{-1}  [\nu]_{(1,\infty)}$ is
subexponential, from which it immediately follows that 
\begin{equation} \label{e:Xsubexp}
  \PP(X_n>x)\sim \nu(x,\infty)= : \bar \nu(x) \ \ \text{as $x\to\infty$}
\end{equation}
and, in particular, $X_n$  has a subexponential distribution; see
\cite{embrechts:goldie:veraverbeke:1979}. We will, in fact, impose a
number of additional assumptions on the L\'evy measure $\nu$. These
assumptions will guarantee that the tail of $X_n$ is light enough to
be in the maximum domain of attraction of the Gumbel distribution. On
the other hand, they will also guarantee that this tail is not ``too
light''.

\begin{assumption} \label{ass;nu}
\phantom{blank}
The distribution $(\nu(1,\infty))^{-1} [\nu]_{(1,\infty)}$ is
both subexponential and in the maximum domain of attraction of the
Gumbel distribution. Furthermore,  there is a distribution  $H_{\#}$
satisfying $\bar\nu(x)\sim a\overline{H_{\#}}$ for $a>0$, and which 
satisfies \eqref{eq;VM1}
  with $c\equiv 1$, i.e.
    \begin{align} 
          &\overline{H_{\#}}(x)= \exp \left ( -\int_{x_0} ^ x
            \frac{1}{h(u)}du    \right ) \text{ for } x>
            x_0\,,\label{eq;f2}    
    \end{align}
    and the auxiliary function $h$  with $h^\prime>0$ on
    $(x_0,\infty)$, and such that 
     \begin{equation} \label{e:quasi.l1}
      \lim_{b\downarrow 1}\limsup_{x\to\infty}\frac{h(bx)}{h(x)}=1\,.
      \end{equation} 
Denoting   
\begin{equation} \label{eq;G}
  G(x) := \left(  \frac{1}{1 - H_{\#}}  \right ) ^ \leftarrow(x), \
  x\geq 0\,, 
\end{equation}
we assume that the function $G$ is of the form
\begin{equation} \label{e:rep.G}
G(x)=   \exp \left \{  \int_{e} ^ x  \frac{\zeta(u)}{u \log u} du
\right \} ,\ x>x_1, \ \ \text{for some $x_1>e$,}
\end{equation} 
where  $\zeta$  satisfies the following
assumptions. 
\begin{enumerate} 
\item[(B1)] $\zeta$ is roughly increasing, i.e.,  
$$  \zeta (x) \asymp \sup_{[1,x]} \zeta (u)   . $$

\item[(B2)] There exists some  $\delta>0$ such that
$$
(\log \log u ) ^ \delta
\ll \zeta (u) \lesssim
\frac{\log u}{\log \log u} .
$$

\item[(B3)]  For the $\delta>0$ in (B2) and for all small $\rho>0$,
$$
 \zeta \left(x^{1-\rho \big / (\log\log x) ^ {\delta \wedge 1}  } \right )\gtrsim \zeta(x)\,.
$$
 
\item[(B4)] For any $c>0$,  
$$
\liminf_{x \to \infty} \int_{x^{1 -c  / \zeta(x)  }} ^ x 
\frac{\zeta(u)}{u \log u } d u  > 0.
$$
\end{enumerate} 

\end{assumption}

We check in Appendix \ref{app;supp} below that the following 
 two important classes of L\'evy measures with subexponential tails
 satisfy Assumption \ref{ass;nu}. 
\begin{example}[lognormal-type tails] \label{eg;lognm1}
\phantom{blank}

$$
\overline{\nu} (x) \sim c x ^ \beta (\log x )^\xi 
\exp\left(   -\lambda (\log x )^ \gamma  \right ) \ \ \text{as $x\to\infty$}
$$
for some $\gamma >1$, $\lambda, c>0$ and 
$\beta,\xi \in \mathbb{R}$. 
\end{example}

\begin{example}[super-lognormal-type tails] \label{eg;suplognm1}
  \phantom{blank}

$$
\overline{\nu} (x) \sim c x ^ \beta (\log x )^\xi 
\exp\left( \lambda (\log x )^ \gamma  \right )
\exp \left(  - \rho  \exp \left( \mu (\log x)^\alpha  \right)
\right )  \ \ \text{as $x\to\infty$}
$$
for some $\alpha \in (0,1)$, $c,\mu, \rho >0$  and
$\beta,\xi,\lambda,\gamma \in \mathbb{R}$. 
\end{example}

\begin{remark}
  \phantom{blank}

The  semi-exponential-type tails such as 
$\overline{\nu}(x) \sim \exp (  - x^\alpha ), \, 0 < \alpha < 1$,
unfortunately, do not satisfy the assumptions and, hence, are 
excluded from our analysis.      
\end{remark}

The following proposition, proved in Appendix \ref{app;supp},
lists certain properties of L\'evy measures satisfying  Assumption
\ref{ass;nu}. We will find these properties useful in the sequel. Let
$\delta>0$ as in Assumption \ref{ass;nu} $(B2)$. 

\begin{proposition} \label{prop;MTG4}
\phantom{blank}

\begin{enumerate} [label=(\roman*)]
 \item       $ G(x) \gg \exp \left\{ (\log \log u )^{1+
        \delta}/(1+\delta)  \right \} $. 
\item     For  any $\alpha_1 > \alpha_2 > 0$, for any $b<\log \alpha_1 - \log \alpha_2$, 
\begin{equation} \label{eq;MTG4.3}
\frac{G(x^{\alpha_1}  )}{G(x^{\alpha_2} )} 
\gg
\exp \left\{   b 
 ( \log \log x  ) ^ \delta  
\right \}\,.
\end{equation}
\item  For  any $H_i \in \text{RV}_{\alpha_i}$, $i=1,2$, 
$\alpha_1 > \alpha_2 > 0$, for any $b<\log \alpha_1 - \log \alpha_2$, 
\begin{equation} \label{eq;MTG4.4}
\frac{h\circ G(H_1 (x)  )}{h\circ G( H_2(x) )} 
\gg \exp \left\{ b ( \log \log x  ) ^ \delta  
\right \}\,.
\end{equation}

\item   For any $\alpha \not=0$, 
    \begin{equation} \label{eq;MTG4.5}
        \left \vert 
        G(x (\log x )^\alpha ) - G(x )
        \right \vert  \asymp  
       ( \log \log x  ) h \circ G (x) \,.
    \end{equation} 
\item  For all sufficiently small  $\rho>0$, 
\begin{equation} \label{eq;MTG4.6}
   \min_{1\leqslant j \leqslant \rho \log x / \zeta(x)}  \frac{G(x) -
     G\left( x  2^{-j} \right )}{j}
\gtrsim  j h\circ G (x)\,.
\end{equation}
 \end{enumerate}
\end{proposition}

\section{Extremal limit theorems} \label{sec;elt}

Let $(X_t)_{t\in\bbz}$ be a stationary infinitely divisible process
\eqref{eq;seqX}, associated with an irreducible, aperiodic, and
null recurrent Markov chain on  $\mathbb{Z}$. Recall that we assume
that the function $f$ is the indicator function \eqref{e:f.ind}. 
Our main result in this
section is a limit theorem for the sequence of random sup-measures
defined by the process via \eqref{e:sup.m.n}. The L\'evy measure $\nu$
of the infinitely divisible random measure $M$ in \eqref{eq;seqX} is
assumed to satisfy Assumption \ref{ass;nu}. We denote
\begin{equation} \label{e:V}
V(x) =  \left(  1/\bar\nu \right ) ^ \leftarrow(x), \
x> 0\,. 
\end{equation}
The Markov chain underlying the process $(X_t)_{t\in\bbz}$ is 
assumed to satisfy Assumption \ref{ass;YRV0}. We define for
$n=1,2,\ldots$ 
\begin{equation} \label{e:bn.X}
  b_n  = V(w_n) +  V\bigl( 1/\overline{F} (n)\bigr), \ \ a_n  = h\circ V(w_n)\,,
\end{equation}
with $w_n$ is the wandering rate in \eqref{eq;wdr}, $F$ is the first
return time law in
\eqref{eq;varphi1}, and $h$ the auxiliary function in \eqref{eq;VM1}. 
 
\begin{theorem} \label{thm;main}
\phantom{blank}

Assume that Assumption \ref{ass;nu} holds, and that Assumption
\ref{ass;YRV0} is satisfied with $0<\beta<1/2$. If $(a_n), (b_n)$ are
given by \eqref{e:bn.X}, then 
\begin{equation} \label{eq;main1}
    \frac{M_n(\cdot)-b_n}{a_n} \Rightarrow \mathcal{M}(\cdot)
\end{equation}
weakly in the space of sup-measures on $[0,1]$, where $(M_n)$
are the random sup-measures in \eqref{e:sup.m.n}, and the limiting random
sup-measure $\mathcal{M}$ is given by \eqref{eq;limRSM1}. 
\end{theorem}

There is a natural counterpart of Theorem \ref{thm;main} that
establishes an extremal limit theorem in a function space. Recall that
a standard Gumbel extremal process is a 
nondecreasing process $\bigl( X_{\rm G}(t),\, t> 0\bigr)$  
satisfying
$$
\PP\bigl( X_{\rm G}(t_i)\leq x_i, \, i=1,\ldots, k\bigr) =\exp\left\{
  -\sum_{i=1}^k(t_i-t_{i-1})e^{-x_i}\right\} 
$$
for $0<t_1<\ldots<t_k$ and $x_1\leq \ldots\leq x_k$. The process
is continuous in probability and has a version in
$D(0,\infty)=\cap_{\vep>0}D[\vep,\infty)$; see
\cite{resnick:rubinovitch:1973}. It is immediate from the definition
of the random
sup-measure $\mathcal{M}$ in \eqref{eq;limRSM1} that
\begin{equation}\label{e:supm.to.extremal}
\bigl( \mathcal{M}([0,t]), \, t>0\bigr) \eid \bigl( X_{\rm G}(t^{1-\beta}) \,;
t>0\bigr)\,,
\end{equation}
see also \cite{lacaux:samorodnitsky:2016}. 
Note that the finite-dimensional convergence part in the following
theorem already follows from Theorem \ref{thm;main}. 
 \begin{theorem} \label{thm;extremal}
   \phantom{blank}
 
Under the assumptions of Theorem \ref{thm;main}, 
\begin{equation} \label{eq;main1}
    \left(\frac{\max_{s\leq nt}X_s-b_n}{a_n}, \, t>0\right)\Rightarrow
    \bigl( X_{\rm G}(t^{1-\beta}),\, t> 0\bigr) 
\end{equation}
weakly in the Skorohod $J_1$ topology on $D(0,\infty)$. 
\end{theorem}

\begin{remark}\label{rk:key}
  {\rm 
Let us return to the discussion in Introduction of this paper and
compare the   statement of Theorems \ref{thm;main} and 
\ref{thm;extremal} 
to the ``expected behavior'' of the extreme values presented in
\eqref{e:Mn.clusterL}. The results of \cite{lacaux:samorodnitsky:2016}
and \cite{samorodnitsky:wang:2019} in the case of regularly varying
tails suggest that $m_n=w_n$,  and the centering and the 
normalization in \eqref{e:Mn.clusterL}  do not appear to be
consistent with Theorema \ref{thm;main} and 
\ref{thm;extremal} due to the presence of an
extra term $V\bigl( 1/\overline{F} (n)\bigr)$ in the centering
sequence. It turns out, however, that  as
long as the tails of the process $(X_t)_{t\in\bbz}$ are ``not too
light'' we have
\begin{equation} \label{e:no.second}
  \lim_{n\to\infty}\frac{V\bigl( 1/\overline{F} (n)\bigr)}{a_n}=0\,,
\end{equation}
and so \eqref{e:Mn.clusterL} does predict the correct centering and the 
normalization. Once the tails of the process become lighter, however,
\eqref{e:no.second} may fail, and a different centering becomes
necessary. We can see this phenomenon on Examples \ref{eg;lognm1}
and \ref{eg;suplognm1}. In fact, for the  lognormal-type tails of Example
\ref{eg;lognm1} the relation \eqref{e:no.second} holds, while for the
super-lognormal-type tails of Example \ref{eg;suplognm1},
\ref{eg;lognm1} holds if $0<\alpha<1/2$ and fails if
$1/2<\alpha<1$. These claims are verified in Appendix \ref{app;supp}. 
}
\end{remark}

We will prove  the two theorems in the remainder of this
section, beginning with Theorem \ref{thm;main}. 
We start with a preliminary analysis that will split the
proof into several steps.  First of all, 
by   (\ref{eq;rsmWC2}),  we
need to prove that  for arbitrarily disjoint open intervals   $  I_1,
\ldots, I_m $ in $ [0,1]$, 
\begin{equation} \label{eq;main1.1}
    \left( \frac{ M_n (I_i)  - b_n }{a_n}  \right)_{i=1,\ldots,m} \Rightarrow 
    \left( \mathcal{M}(I_i)  \right)_{i=1,\ldots,m}
  \end{equation}
  weakly in $\bbr^m$. Note, further, that for any $0<\vep<a$ 
the function $V$ in \eqref{e:V} satisfies
\begin{equation} \label{e:V.G}
  G\bigl(x(a-\vep)\bigr)\leq V(x) \leq G\bigl(x(a+\vep)\bigr)
\end{equation}
for all $x$ large enough. Next, we decompose the  stationary process
$(X_t)_{t\in\bbz}$ as 
follows. Let $M^{(1)}$ and $M^{(2)}$ be two independent infinitely
divisible random
measures on $(E,\mathcal{E})$, both with with the same control measure
$\mu$ as the measure $M$ in \eqref{eq;seqX}. 
 With $(\sigma^2, \nu,b )$ being the local
characteristic triple of $M$, we set the local characteristic triple
of $M^{(1)}$ to be $(0, [\nu ]_{(x_0,\infty)},0)$, and  the local
characteristic triple of $M^{(2)}$ to be 
$(\sigma^2, [\nu ] _{(-\infty,x_0]},b )$, with $x_0$ as in
\eqref{eq;f2}.  If we define for   each $t\in \mathbb{Z}$ 
\begin{equation} \label{eq;main1.2}
    X^{(1)}_t  = \int_E f\circ\theta^t(x)\, M^{(1)}(dx) , \quad 
    X^{(2)}_t  = \int_E f\circ\theta^t(x)\, M^{(2)}(dx)\,,
\end{equation}
then  $\{ X^{(i)}_t  \} _{t \in \mathbb{Z}}, i=1,2$ are two
independent stationary infinitely divisible processes such that 
    $ \{ X_t  \} _{t \in \mathbb{Z}} \stackrel{d}{=}  \{X^{(1)}_t  +
    X^{(2)}_t\}_{t \in \mathbb{Z}}$. 
For $i=1,2$ we let  $ M^{(i)}_n (\cdot)$  be the random sup-measure defined for 
$\{X^{(i)}_t\}_{t\in \mathbb{Z}}$   as in \eqref{e:sup.m.n}. The
following proposition shows that $M_n^{(2)}$ is asymptotically
negligible with our scaling. 

\begin{proposition} \label{prop;X2n} $M^{(2)}_n([0,1])/a_n\to 0$ as
  $n\to\infty$. 
\end{proposition}
\begin{proof}
 Since the L\'evy measure of $X^{(2)}_0$ is bounded on the right, 
$\mathbb{P}(X^{(2)}_0 > r )  =o( e^{-c r})$ for any  $c>0$  see
e.g. Theorem 26.1 in \cite{sato:2013}. Using the fact that
$\zeta(x)\to\infty$ we use \eqref{e:V.G} and part (i) of Proposition
\ref{prop;MTG4} to see that for any $p>0$,
for all  large $n$, 
\begin{align*}
a_n &\geq h\circ G(aw_n/2) =  \frac{G(aw_n/2)\zeta(aw_n/2)}{\log
      (aw_n/2)} \\
  &\gg \frac{G(aw_n/2)}{\log (aw_n/2)} \gg (\log w_n)^p\,,
\end{align*}
therefore, taking $p>1$ we have for any $\epsilon>0$
\begin{align*}
\mathbb{P} \left( M^{(2)}_n ([0,1] )> \epsilon a_n \right ) \leq &n
  \mathbb{P} \left( X^{(2)}_0 > \epsilon (\log w_n)^p \right ) \\
=&o\left(n  \cdot \exp\left\{ -\epsilon (\log w_n)^p \right\}\right ) \to 0
\end{align*}
by \eqref{e:wn}.  
\end{proof}

Proposition \ref{prop;X2n} implies that, in order to show 
 (\ref{eq;main1.1}), we need to prove that 
\begin{equation} \label{eq;main1.4}
     \left( \frac{ M_n ^{(1)}(I_i)  - b_n }{a_n}  \right)_{i=1,\ldots,m} \Rightarrow 
    \left( \mathcal{M}(I_i)  \right)_{i=1,\ldots,m}  \,,
\end{equation}
which we now carry out.  Consider a probability space (which we will
denote, with some abuse of notation, by $\ProbSpace$) supporting
i.i.d. random elements $\{ Y^{(j,n)}  \}_{j\in \mathbb{N}}$
distributed with the law $\mu_n$ in \eqref{eq;mun} for each 
$n=1,2,\ldots$, as well as i.i.d. random closed subsets of $[0,1]$,
$\{  \overline{R_j}\}_{j\in \mathbb{N}}$ distributed as in
(\ref{eq;tctSRS}) such that, with $I_{j,n}$ defined by \eqref{eq;Ijn}
we have 
\begin{equation}  \label{eq;main1.5}
   \frac{1}{n} I_{j,n}   \to 
   \overline{R_j} \quad \text{a.s. as} \ \ n \to \infty \ \ \text{for
     each} \ \ j\in \mathbb{N};
 \end{equation}
this is possible by \eqref{eq;mtg3.1} and the Skorohod embedding. The
same probability space also supports a sequence $\{\Gamma_j \}_{j\in
  \mathbb{N}}$ of the arrival times of a unit rate Poisson process on
$\bbr_+$, independent of  $\{ I_{j,n}, \overline{R_j} : j, n \in
\mathbb{N} \}$.  The following series representation of the process
$\{X^{(1)}_t \}_{t \in \mathbb{Z}}$ is the key for our argument. It
follows from Theorem 3.4.1 in \cite{samorodnitsky:2016}. For each
$n\in \mathbb{N}$, 
\begin{equation}\label{eq;main1.7}
    \left( X^{(1)}_t \right )_{0 \leqslant t \leqslant n}
    \stackrel{d}{=} \left( \sum_{j=1} ^ \infty
    \tilde V\bigl( w_n/\Gamma_j  \bigr)\one_{\{ t\in I_{j,n} \} }\right)_{0
    \leqslant t \leqslant n}  \,, 
\end{equation}
where
\begin{equation} \label{e:V.tilde}
\tilde V(y) = \left\{ \begin{array}{ll}
                        V(y) & \text{for $y>1/\bar\nu(x_0)$} \\
                        0 & \text{otherwise}
                      \end{array}
                    \right..
\end{equation}                 
 When proving (\ref{eq;main1.4}) we will simply assume that the process
$\{X^{(1)}_t \}_{t \in \mathbb{Z}}$  is given by the right hand side
of \eqref{eq;main1.7}. Furthermore, for notational simplicity we will
drop the ``tilde'' over $V$ in the sequel, while keeping in mind that
it vanishes for small values of the argument, as in
\eqref{e:V.tilde}. We now state several propositions that will
prove (\ref{eq;main1.4}).  
 
For $k\in \mathbb{N}$ we define, in the notation of  (\ref{eq;hatI1})
and (\ref{eq;hatIj}), 
 \begin{align*} 
        & M_{n,(k)} (B) = 
 \max_{t\in nB\cap \widehat{I}_{k,n}} \sum_{j=1}^\infty V \bigl( w_n/\Gamma_j\bigr)
\one_{\{ t \in {I}_{j,n}  \} },   \\ 
        & \mathcal{M}_{(k)} (B) = \left\{ 
\begin{array}{ll}
- \log \Gamma_k  & \text{if } \ \ \overline{R_k} \cap B \neq \emptyset \\
-\infty & \text{otherwise}
\end{array}  \right. .
 \end{align*}
 
 \begin{proposition}  \label{prop:new.k}
    For each $k\in \mathbb{N}$ and each open interval $I$ in $[0,1]$,
    \begin{equation} \label{eq;ELT1}
    \frac{M_{n,(k)}(I) - b_n}{a_n} \stackrel{P}{\longrightarrow}
    \mathcal{M}_{(k)} (I) . 
  \end{equation}
\end{proposition}

We define, further, for $K\in \mathbb{N}$,
  \begin{align*} 
        & M_{n,K} (B)  =  \bigvee_{k=1}^ K M_{n,(k)}(B), \\
       & \mathcal{M}_K (B)  = \bigvee_{k=1}^ K 
        \mathcal{M}_{(k)} (B). 
  \end{align*}
 It follows  from Proposition \ref{prop:new.k} that 
for each $K$ and each open interval $I$ in $[0,1]$, 
 \begin{equation} \label{e:up.toK}
    \frac{M_{n,K}(I) - b_n}{a_n} \stackrel{P}{\longrightarrow}
    \mathcal{M}_{K} (I) \,.
    \end{equation}
 Since it is also clear that for any open interval $I$ in $[0,1]$, as
 $K\to\infty$, 
  $$
    \mathcal{M}_K(I) \ \longrightarrow
    \mathcal{M} (I) \ \ \text{a.s.}
    $$
 if the limting sup-measure $\mathcal{M}$ is defined on the same
 probability space $\ProbSpace$ by \eqref{eq;limRSM2}, then 
the only remaining step to establish  (\ref{eq;main1.4}) is the
following claim. 
\begin{proposition} \label{prop;ELT(2)}
 For any open interval $I$ in $[0,1]$ and $\epsilon >0 $, 
\begin{equation} \label{eq;ELT4}
    \lim_{K \to \infty} \; \lim_{n \to \infty}
    \mathbb{P} \left(
    \left \vert 
    M_{n,K}(I) - M_n (I) 
    \right \vert \geqslant \epsilon 
    \right ) = 0 \,.
\end{equation}
\end{proposition} 

We now prove Propositions \ref{prop:new.k}  and \ref{prop;ELT(2)}. 
\begin{proof} [Proof of Proposition  \ref{prop:new.k}]
 \phantom{blank}
 
We will only consider the case $I=(a,b)$ for some $0\leqslant a < b
\leqslant 1$, the other cases being similar. Note that, by
\eqref{eq;mtg3.3} it is enough to prove the proposition for
$k=1$. Since the normalized tail $\bar\nu$ is in the Gumbel maximum
domain of attraction, the function $V$ is $\Pi$-varying, so by
\eqref{e:pi} and \eqref{e:small.diff}  we have 
 \begin{equation} \label{e:first.t}
      \frac{V(w_n / \Gamma_1) - V(w_n)}{h\circ V(w_n)}
       \longrightarrow - \log \Gamma_1 \ \ \text{a.s. as
         $n\to\infty$. }
\end{equation}
Furthermore, by (\ref{eq;main1.5})
$$
     \one_{\{ I_{1,n} \cap n I \neq \emptyset  \} } 
     \longrightarrow 
     \one_{\{ \overline{R_1} \cap  I \neq \emptyset  \} }  \ \ \text{a.s. as
         $n\to\infty$. }
$$
Since 
 $ h \circ V(w_n)  = o( V(w_n) )$, we have 
 \begin{align*}  
     &\frac{
     V\left( w_n / \Gamma_1  \right) 
     \one_{\{ I_{1,n} \cap n I\neq \emptyset  \} }      -  V(w_n) }
     {h\circ V(w_n)}   \longrightarrow 
     \mathcal{M}_{(1)}(I)  \ \ \text{a.s..}
 \end{align*}
If we denote 
$$
     S_{n,(1)} (I) = M_{n,(1)}(I) - 
      V \bigl( w_n/\Gamma_1  \bigr) \one_{\{ I_{1,n} \cap n I \neq \emptyset  \} }\,,
$$
then the claim of the proposition will follow from the following two statements:
 \begin{equation}  \label{eq;top1.3} 
     \limsup_{n \to \infty} \frac{S_{n,(1)}(I)  - V\bigl( 1 /\overline F(n)\bigr)}{ h\circ V(w_n) }
     \leqslant 0  \quad \text{in probability} 
 \end{equation}    
and 
\begin{equation}  \label{eq;top1.4}  
  \mathbb{P} \left(
    \liminf_{n \to \infty} 
     \frac{S_{n,(1)}(I)  - V\bigl( 1/\overline F(n)\bigr)}{ h\circ V(w_n) }
\geqslant 0 
    \, \Big \vert \, 
    \overline{R_1} \cap I \neq \emptyset
    \right ) = 1 \,,
\end{equation}
 which we proceed to prove. We start with (\ref{eq;top1.3}). Note that 
 $$
 0 \leqslant S_{n,(1)} (I) \leqslant
 \max_{t \in  I_{1,n} } \sum_{j = 2} ^ \infty  
  V \bigl( w_n/\Gamma_j  \bigr) 
  \one_{\{ t \in I_{j,n}  \} } =: S_{n,(1)} .
 $$
Let $c_1, c_2>0$ be positive constants to be determined later and write
$A_{c_1,n} = c_1  \log n / \overline F(n) $. Then 
 \begin{align*}
    & \mathbb{P}\left( S_{n,(1)}  \geqslant 
     V\left (A_{c_1,n}  \right )  + c_2 h \circ V\left ( A_{c_1,n} \right ) \right ) \\
    \leqslant & A_{c_1,n} \cdot 
    \mathbb{P}\left(\sum_{j=2}^\infty 
    V \bigl( w_n/\Gamma_j  \bigr) 
    \one_{ \{  0 \in I_{j,n}  \}  }  \geqslant  
    V\left ( A_{c_1,n}  \right )  + c_2 h \circ V\left ( A_{c_1,n}  \right )
    \right )  \\ 
     &+ \mathbb{P} \left( \# I_{1,n} > A_{c_1,n}  \right) \\
     \leqslant & A_{c_1,n} \cdot 
    \mathbb{P}\left( X_{0,n}  \geqslant  
    V\left ( A_{c_1,n}  \right )  + c_2 h \circ V\left ( A_{c_1,n}  \right )
    \right )  
     + \mathbb{P} \left( \# I_{1,n} >  A_{c_1,n}  \right) 
     \\  =:&  A_{c_1,n} \cdot B_1 + B_2 . 
\end{align*}
 By    (\ref{eq;obsB1.3})  $B_2 \to 0$  if $c_1$ is large
 enough. Further, 
 $A_{c_1,n} \cdot B_1 \to e^{-c_2}$ by Proposition 0.9 in 
 \cite{resnick:1987}. Therefore, fix any 
 $\epsilon \in (0,1)$, we can choose  $c_1, c_2 >0$ such that 
 $$
 \mathbb{P}\left( S_{n,(1)}  \leqslant 
     V\left (A_{c_1,n}  \right )  + c_2 h\circ V\left ( A_{c_1,n} \right ) \right ) \geqslant 1 - \epsilon. 
 $$
 The claim   (\ref{eq;top1.3}) follows by
 \eqref{e:small.diff} and parts (ii),  (iii), (iv) of Proposition
 \ref{prop;MTG4}, 
 \begin{align*}
      V(A_{c_1,n}) - V \bigl(  1/\overline F(n) \bigr) 
      & = G(A_{c_1,n}) -   G \bigl(  1/\overline F(n) \bigr) \\
     & + o (h\circ G(A_{c_1,n})) + 
     o \left(h \circ  G \bigl(  1/\overline F(n) \bigr)
     \right )    \\
     & \lesssim  ( \log \log n ) h \circ G \bigl(  1/\overline F(n) \bigr) \\ 
     &  + o (h\circ G(A_{c_1,n})) + 
     o \left(h \circ  G \bigl(  1/\overline F(n) \bigr)
     \right ) \\
     &  =  o (h \circ V(w_n)) \,.
 \end{align*}

 We now prove (\ref{eq;top1.4}).   Let $\Omega_1= \{  I_{1,n} \cap n I
 \neq \emptyset  \}$. For  a fixed $\omega_1 \in \Omega_1$ we view 
   $ \left \{ \one_{ \{  I_{1,n} (\omega_1 )   \cap  
  I_{j,n} \cap n I \neq \emptyset \} }:  j=2,3,\ldots \right \} $ as
a Bernoulli sequence with the success probability 
$\overline{p}_{n,T}(\omega_1)$ in (\ref{eq;quench}). By Theorem
\ref{thm;istP}  $(\rom{2})$ and $(\rom{3})$, for every $0<\epsilon<1$
 and $\gamma > (1-2\beta)^{-1}$    we can choose new  $c_1,c_2>0$  
   such that the event 
  \begin{align*} 
    & D_1 : = \left\{
   \frac{c_1   L( (\log n)^\gamma )}{w_n \overline F(n) (\log n)^{\gamma
      \beta}} \leqslant \overline{p}_{n,T}  
   \leqslant \frac{c_2  \log n}{w_n \overline F(n)}
   \right \}  
  \end{align*}
satisfies $\mathbb{P} \left( D_1 \, \vert \, \Omega_1 \right )
   \geqslant 1 - \epsilon$ for all $n$ large enough. 

For
   $\omega_1\in\Omega_1$ we denote  
 $j_1=j_1(\omega_1) =\inf \{ j \geqslant 2 : 
I_{j,n} \cap I_{1,n}(\omega_1) \cap n I \neq \emptyset \}$ and note
that $S_{n,(1)} (I) \geqslant V(w_n / \Gamma_{j_1})$. Therefore, for
any $c_3>0$ we have 
\begin{align*}
&\mathbb{P} \left(    
S_{n,(1)} (I) \geqslant V\left(  
\frac{c_1}{c_3}  \cdot \frac{  L( (\log n) ^ \gamma ) }{ \overline F(n) (\log n) ^ {\gamma \beta} }
\right)
\,\Big \vert \,    \overline{R_1} \cap I \neq \emptyset \right ) \\
\geqslant &\mathbb{P} \left( D_1\cap\left\{ 
 V(w_n / \Gamma_{j_1})\geqslant V\left(  
\frac{c_1}{c_3}  \cdot \frac{  L( (\log n) ^ \gamma ) }{ \overline F(n) (\log n) ^ {\gamma \beta} }
\right) \right \}  
\,\Big \vert \,    \overline{R_1} \cap I \neq
            \emptyset   \right )  \\
\geqslant & \mathbb{P} \left( D_1\cap\left\{ \Gamma_{j_1}\leqslant
            c_3(\overline{p}_{n,T} ){^ {-1} }   \right\}   \,\Big \vert \,    \overline{R_1} \cap I \neq
            \emptyset        \right ) \\
\geqslant & \mathbb{P} \left( D_1\cap\left\{  j_1\leqslant
            (c_3/2)(\overline{p}_{n,T} ){^ {-1} } \right \} \,\Big \vert \,    \overline{R_1} \cap I \neq
            \emptyset     \right ) -\epsilon \\
\geqslant & \mathbb{P} \bigl( D_1\, \vert \,    \overline{R_1} \cap I \neq
            \emptyset\bigr) -2\epsilon \geqslant  
\mathbb{P} \bigl(  D_1\, \vert \,  \Omega_1)-3\epsilon \geqslant
            1-4\epsilon\,, 
\end{align*}
for large $n$, 
where the 3rd inequality follows from the law of large numbers, the 4th
inequality follows from the Markov inequality if $c_3$ large enough,
and the penultimate inequality follows from \eqref{eq;main1.5}. Since
we can take $\epsilon$ as small as we wish,  it is enough to show that
\begin{equation} \label{eq;top3.6}
    \left \vert V\left(  
 \frac{c_1}{c_3}  \cdot \frac{  L( (\log n) ^ \gamma ) }{ \overline F(n) (\log n) ^ {\gamma \beta} }
\right)  - V \left(  1/\overline F(n) \right )
\right \vert  = o \left( h\circ V(w_n)  \right) \,.
\end{equation}
To this end, choose any $\alpha> \gamma\beta$ and note that for large $n$,
by parts $(\rom{3})$ and $(\rom{4})$ of Proposition \ref{prop;MTG4},
the expression in the left-hand side does not exceed 
\begin{align*}
 &V \left(  1/\overline F(n) \right )   -   V \left(  (\log
  n)^{-\alpha}/\overline F(n)\right ) \\
    \lesssim &  (\log \log n ) \, h \circ  V \left(  1/\overline F(n) \right)
\ll h \circ V(w_n) \,, 
\end{align*}
as required.   
\end{proof}

\begin{proof}[Proof of Proposition \ref{prop;ELT(2)}]
We start by fixing  a small constant $\rho$ and setting 
\begin{equation}  \label{eq;kn}
    i_ n =  \left \lfloor \frac{ \rho \log w_n } {   \zeta(w_n) }  \right \rfloor.
  \end{equation}
The first step is to establish the following claim, that 
shows that for large $k$,  $M_{n,(k)}(I)$ is not
likely to become the overall maximum $M_n(I)$.   
\begin{equation}   \label{eq;mid1}
    \lim_{i_0 \to \infty, \, K \to \infty} \; 
    \limsup_{n \to \infty}
    \mathbb{P} \left(
    \max _ {  2^{i_0} \leqslant k < 2^{i_n} } 
    M_{n,(k)}(I)
     >  M_{n,K}(I)
    \right ) = 0\,.
\end{equation}

To this end, observe that, by Proposition \ref{prop:new.k}, for any
$ \epsilon \in (0,1)$ we can choose $C_1>0$ large enough so that for
all $K$ large enough, 
$$
\lim_{n \to \infty}  \mathbb{P} \left (  M_{n,K} (I) \geqslant b_n -
  C_1 a_n  \right)  \geqslant 1-\epsilon\,.
$$
Next, for $c>0$ let $A_{c,n}= c   \log n / \overline F(n)$. For  $i_0
\leqslant j \leqslant i_n-1$, let $H_j$  be the event  
$$
   \bigcap_{k=2^j} ^ {2^{j+1}-1} \left\{ 
    M_{n,(k)}(I)\leqslant V \bigl(  2w_n/k  \bigr) + 
    V \left( A_{c,n}  \right) + 2 j h \circ V (A_{c,n})
    \right \} .
$$
We claim that, given $0<\epsilon<1$, we can find $c> 0$
such that 
\begin{equation}  \label{eq;mid1.3}
 \lim_{i_0 \to \infty} \, \liminf_{n \to \infty}  \,   \mathbb{P}\left(  \bigcap_{j=i_0 }
    ^{ i_n - 1 } H_j 
    \right)  \geqslant 1 - \epsilon.
\end{equation}
Assuming, for a moment, that this is true, the claim \eqref{eq;mid1}
will follow once we check that for all $j$ and $n$ large enough, 
\begin{equation} \label{eq;mid1.4}
   V\left( w_n \right) - V \bigl(w_n/2^{j-1}  \bigr) - C_1 a_n >
   V(A_{c,n}) + 2 j h\circ V(A_{c,n}) - V(1/ \overline F(n) ) . 
\end{equation}
Indeed, by \eqref{e:small.diff} and part $(\rom{5})$ of Proposition
\ref{prop;MTG4}, 
\begin{align*}
    &  V ( w_n ) -  V \bigl(w_n/2^{j-1}  \bigr)  - C_1 a_n \\
    = & G ( w_n ) - G \bigl(w_n/2^{j-1}  \bigr)   - ( C_1 +
        o(1) ) h \circ G(w_n)     \\
  \gtrsim & \, j  h \circ G(w_n)\,,
\end{align*}
while by part $(\rom{4})$ of Proposition
\ref{prop;MTG4}, 
\begin{align*}
    & V(A_{c,n}) + 2 j h \circ V(A_{c,n}) - V(1 / \overline F(n) ) \\
    = & G(A_{c,n}) + 2 j h \circ G(A_{c,n}) - G(1 / \overline F(n) ) + o ( j 
        h \circ G(A_{c,n}) )   \\
     \lesssim & \, (j + \log \log n) h \circ V(A_{c,n})\,.
\end{align*}
By part $(\rom{3})$ of Proposition \ref{prop;MTG4} this gives 
 (\ref{eq;mid1.4}), and, hence, (\ref{eq;mid1}), so we now prove
 (\ref{eq;mid1.3}). Switching to the complements, we will show that 
\begin{equation}  \label{eq;mid2.1}
  \lim_{i_0 \to \infty} \,  \limsup_{n \to \infty}  \,  \mathbb{P}\left(  
    \bigcup_{j=i_0} ^ {i_n -1} H_j ^ c 
    \right ) \leqslant \epsilon \,.
\end{equation}
Recall that 
$$    M_{n,(k)} \leqslant 
    V\bigl( w_n/\Gamma_k  \bigr ) + S_{n,(k)} , \quad
    S_{n,(k)} := \max_{t\in I_{k,n} } \sum_{j=k+1} ^ \infty
    V\bigl( w_n/\Gamma_j \bigr) 1_{\{ t \in I_{j,n}  \} } .
$$
Therefore, for each $i_0 \leqslant j < i_n$,
$
H_j ^ c \subseteq \bigcup_{k=2^j} ^ {2^{j+1} -1 }
\left(  U_k  \cup D_k  \cup L_k    \right )
$
with 
\begin{align*}
    &U_k  = \left\{ \Gamma_k \leqslant k/2   \right \},  \\
    &D_k = \left\{ \# I_{k,n} > A_{c,n}  \right \} ,\\
    & L_k = \left \{
    S_{n,(k)} \geqslant  V(A_{c,n}) + 2 j h \circ V(A_{c,n}),\, 
     \# I_{k,n} \leqslant A_{c,n} 
    \right \} . 
\end{align*}
Trivially, 
\begin{equation} \label{eq;mid2.2}
    \sum_{k=1}^\infty \mathbb{P}(U_k) < \infty\,,
\end{equation}
and, if $c$ is large enough, then  by   (\ref{eq;obsB1.3}) we also
have 
\begin{equation}  \label{eq;mid2.3}
   \sum_{k=1}^\infty \mathbb{P}(D_k) < \infty\,.
\end{equation}
Next,  as $\mathbb{P}(X_{0}^{(1)} > x) \sim \overline{\nu}(x)$ by
 subexponentiality  
and $(A_{c,n})^{-1} \asymp \overline{\nu}(V(A_{c,n}))$, we have for
$2^j\leqslant k<2^{j+1}$, 
\begin{align*}
    \mathbb{P}(L_k) \leqslant & A_{c,n} \cdot  \mathbb{P} 
    \left(  X_{0}^{(1)}
    \geqslant 
     V(A_{c,n}) + 2 j h \circ V(A_{c,n})
    \right )   \\
    \lesssim & 
    \frac{ \overline{\nu} \left(     
     V(A_{c,n}) + 2 j h\circ V(A_{c,n})
    \right ) }
    {\overline{\nu} \left(
    V(A_{c,n})
               \right )}\\
  \lesssim & \exp \left\{ - 
\int_0 ^ {2j} \frac{ h\circ V(A_{c,n})  }
{  h\bigl[  V(A_{c,n})  + u h\circ V(A_{c,n} )   \bigr]     }
             du \right \} \\
   \lesssim & \exp \left\{ -\frac{ 2jh\circ V(A_{c,n})  }
{  h\bigl[  V(A_{c,n})  + 2(i_n-1) h\circ V(A_{c,n} )   \bigr]     }
           \right \}\,.
\end{align*}
Note that by Assumption \ref{ass;nu} $(B1)$, for some constant $C$,
for large $n$,
\begin{align*}
     2 (i_n -1) h \circ V(A_{c,n}) \sim &  2 i_n  h\circ G(A_{c,n})      \\
    \sim & \frac{2 \rho \log w_n}{\zeta(w_n)} \cdot 
    \frac{ \zeta(A_{c,n}) }{ \log A_{c,n} }  V (A_{c,n})   
    \leqslant C\rho V(A_{c,n})\,, 
\end{align*}
so we can choose $\rho$ small enough so that
\begin{align*}
 \mathbb{P}(L_k) \leqslant \exp \left\{ -2j\frac{h\circ
  V(A_{c,n})}{h\bigl[ (1+C\rho) V(A_{c,n} )   \bigr]}\right\}\leqslant e^{-j}
\end{align*}
because $h$ is assumed to satisfy \eqref{e:quasi.l1}. It  follows that 
$$
\sum_{k=1}^\infty \mathbb{P}(L_k) < \infty
$$
which, together with (\ref{eq;mid2.2}) and (\ref{eq;mid2.3}), proves
(\ref{eq;mid2.1}), so we have established \eqref{eq;mid1}. Now  the
claim of Proposition \ref{prop;ELT(2)} will follow from the following
statement that we prove next. 

We claim that, with $i_n$ given, once again, by \eqref{eq;kn}, 
$$
  \lim_{K\to\infty} \;  \limsup_{n\to \infty}
   \mathbb{P} \left( 
    \max_{ k \geqslant 2^{i_n} }  M_{n,(k)}   <  M_{n,K}(I)
    \right ) = 1 \,.
$$
Since  $ b_n \sim G(w_n)$ and $a_n = o(b_n)$, by \eqref{e:up.toK} it is enough to show
that for some $\eta \in (0,1)$, 
\begin{equation} \label{eq;btm1.1}
    \lim_{n \to \infty} \mathbb{P} \left(
     \max_{ k \geqslant 2^{i_n} }  M_{n,(k)}   \leqslant \eta G(w_n)
    \right ) = 1\,.
\end{equation}

To this end, choose $0<r<(1-2\beta)/2$ and write 
\begin{align*}
    \max_{ k \geqslant 2^{i_n} } M_{n,(k)} & \leqslant 
    \max_{0 \leqslant t \leqslant n}  \sum_{j=2^{i_n}  } ^ { \lfloor n^{r}  \rfloor } 
    V \bigl(  w_n/\Gamma_j\bigr)
    \one_{ \{ t \in I_{j,n} \} } \\
    & + 
    \max_{0 \leqslant t \leqslant n}  \sum_{j=\lfloor n^{r}  \rfloor +1 } ^ \infty 
       V \bigl(  w_n/\Gamma_j\bigr)
    \one_{ \{ t \in I_{j,n} \} } 
     = : T_{1,n} + T_{2,n}\,.
\end{align*}
By the choice of $r$, 
$$
     \mathbb{P}\left( \max_{0\leqslant t \leqslant n} \sum_{j=2^{i_n}  } ^ { \lfloor n^{r}  \rfloor } 
    \one_{ \{ t \in I_{j,n} \} }  \geqslant 2  \right) 
    \leqslant  n \mathbb{P} \left(  
    \sum_{j=1  } ^ { \lfloor n^r  \rfloor } 
    \one_{ \{ 0 \in I_{j,n} \} }  \geqslant 2 \right ) 
    \lesssim  \frac{n^{2r+1}}{w_n^2} \to 0.
$$
Therefore, with probability increasing   to 1, 
$$
T_{1,n} \leqslant V  \bigl(   w_n/\Gamma_{2^{i_n}}   \bigr)   \lesssim
G \bigl(  w_n/  2^{i_n-1} \bigr)\,.
$$
By Assumption \ref{ass;nu}  $(B4)$,  
$$
 \limsup_{n \to \infty}   \frac{G(w_n / 2^{i_n-1} )}{G(w_n)} 
    <1\,, 
$$
so  (\ref{eq;btm1.1}) will be established once we prove that for any
$\epsilon>0$, 
$$
  \lim_{n \to \infty} \mathbb{P} \left(
  T_{2.n} > \epsilon G(w_n)
  \right ) = 0  \,.
$$
The latter statement will follow from the following claim:
$$
     \lim_{n \to \infty}  n \cdot \mathbb{P} \left (
    \sum_{\Gamma_j > n^r }  V \bigl(  w_n/\Gamma_j   \bigr )
    \one_{\{ 0 \in I_{j,n} \} } > \epsilon G(w_n)
    \right ) = 0\,.
 $$
  Since for some $s>0$ $V(x)\leq G(sx)$ for all $x$, we will prove
  instead that
  \begin{equation} \label{eq;btm1.4}
     \lim_{n \to \infty}  n \cdot \mathbb{P} \left (
    \sum_{\Gamma_j > n^r }  G \bigl( s w_n/\Gamma_j   \bigr )
    \one_{\{ 0 \in I_{j,n} \} } > \epsilon G(w_n)
    \right ) = 0\,.
  \end{equation}

To this end, denote for  $n\in \mathbb{N}$, 
\begin{align*}
    &\widetilde{x}_n = \epsilon G(w_n), \quad x_n = G(sw_n /n^{r}),
      \quad 
     m_0 =  \lfloor   \widetilde{x}_n / x_n \rfloor ,  
\end{align*}
and define $H_n:  (x_0,\infty)^n \to \mathbb{R}$ by 
$$
H_n (z_1, \ldots, z_n) = \int_{x_0} ^ {z_1} \frac{du}{h(u)} +\cdots +\int_{x_0} ^ {z_n} \frac{du}{h(u)} 
:= q(z_1) + \cdots + q(z_n)\,.
$$
Let $N_n$ be a 
Poisson random variable  with mean $ s(1 - n^rs^{-1}w_n^{-1})$
(positive for large $n$). If  $\{\xi_i\}_{i=1}^\infty$ is a 
family of i.i.d. random variables independent of $N_n$ with
distribution equal to normalized $H_{\#}$ restricted to the interval
$(x_0, x_n)$.  Then 
\begin{equation}  \label{eq;btm2.1}
  \sum_{\Gamma_j > n^r }  G \bigl(s w_n/\Gamma_j   \bigr)
    \one_{\{ 0 \in I_{j,n} \} }  
    \stackrel{d}{=}
    \sum_{i=1}^{N_n} \xi_i\,,
\end{equation}
so that 
\begin{align} \notag
    & \mathbb{P} \left (
    \sum_{\Gamma_j > n^r }  G\bigl( sw_n/\Gamma_j   \bigr)
    \one_{\{ 0 \in I_{j,n} \} } > \widetilde{x}_n
    \right ) \\
    & =  \sum_{d=1} ^ \infty 
    \mathbb{P} \left( N_n = m_0+d   \right )
        \mathbb{P} \left(  \sum_{i=1}^{m_0+d} \xi_i
        >  \widetilde{x}_n \right )  
         = : \sum_{d=1}^\infty B_d \cdot Q_d\,.\label{e:last.sum}
\end{align}
Clearly, 
\begin{equation}  \label{eq;btm2.2}
    B_d \leqslant  s^{m_0+d}/(m_0+d)!\,.
\end{equation}
On the other hand, for some constant $c>0$, 
\begin{align*}
 Q_d =& \int_{(x_0 ,x_n)^{m_0+d}}
      \one_{ \{ \sum_{i=1}^{m_0+d} z_i > \widetilde{x}_n  \}  }
        \prod_{i=1}^{m_0+d} P_{\xi_i} (d z_i ) \\
    \leqslant  &c^{m_0+d} \int_{(x_0 ,x_n)^{m_0+d}}
      \one_{ \{ \sum_{i=1}^{m_0+d} z_i > \widetilde{x}_n  \}  }
        \prod_{i=1}^{m_0+d} H_{\#} (d z_i) \\
      = &c^{m_0+d}\int_{(x_0 ,x_n)^{m_0+d}}
      \one_{ \{ \sum_{i=1}^{m_0+d} z_i > \widetilde{x}_n  \}  }
        \prod_{i=1}^{m_0+d}  \exp \{ - q(z_i) \} q^\prime(z_i)\, d z_i \\
      \leqslant  & \left( c q(x_n)  \right)^{m_0+d}  
     \exp \left ( -  \inf \left\{
    \sum_{i=1}^{m_0+d} q(z_i)       : 
    \sum_{i=1}^{m_0+d} z_i > \widetilde{x}_n,  x_0 < z_i <x_n
                   \right \} \right ) \,.
\end{align*}
To evaluate the infimum inside the above exponential,  note that that
the function $H_n$ is increasing and concave in all of its
variables. Hence its infimum is achieved at a boundary point which
will have, say, $k_d$ coordinates equal to $x_n$, $m_0+d-k_d-1$
coordinates equal to $x_0$, and a final coordinate that makes the sum
of all coordinates equal to $\widetilde{x}_n$, for the smallest
possible value of $k_d$ that makes it possible. That means that
\begin{equation} \label{e:inf}
\inf\left\{
    \sum_{i=1}^{m_0+d} q(z_i)       : 
    \sum_{i=1}^{m_0+d} z_i > \widetilde{x}_n,  x_0 < z_i <x_n
  \right \} \geq k_dq(x_n)\,.
\end{equation}
Clearly,
\begin{equation} \label{e:k.d}
k_d=\left[\left\lceil
  \frac{\widetilde{x}_n-x_n-(m_0+d-1)x_0}{x_n-x_0}\right\rceil
\right]_+ \geqslant \frac{\widetilde{x}_n-x_n}{2(x_n-x_0)}
\end{equation}
if
\begin{equation} \label{e:small.d}
d\leqslant \frac{(\widetilde{x}_n-x_n)}{2 x_0}-m_0 + 1\,.
\end{equation} 
Notice that by \eqref{eq;btm2.2}, the part of the sum in
\eqref{e:last.sum} corresponding to $d$ outside of the above range
does not exceed, for large $n$, 
$$
\sum_{d\geqslant \widetilde{x}_n/x_0} s^{m_0+d}/(m_0+d)!=o(1/n)
$$
by part $(i)$ of Proposition \ref{prop;MTG4}. On the other hand, for
large $n$, for $d$ in the range \eqref{e:small.d}, $k_d\geq m_0/3$ by
\eqref{e:k.d}. Therefore, the part of the sum in
\eqref{e:last.sum} corresponding to $d$ in the range \eqref{e:small.d}
can be bounded by 
\begin{align*}
      \sum_{d=1}^\infty \frac{\left( cs q(x_n)
  \right)^{m_0+d}}{(m_0+d)!} \exp \left \{  - m_0 q (x_n)/3   \right
  \}  =\exp\bigr\{ -q(x_n)(m_0/3-cs) \bigr\}\,.
\end{align*}
Since $m_0 \gg \log \log n \to \infty$ by part $(ii)$ of Proposition
\ref{prop;MTG4}, and 
  $$
       q(x_n)  =  - \log \overline{H_\#} (G(sw_n/n^r) ) \sim (1-\beta
       - r ) \log n \,, 
       $$
       the part of the sum in
\eqref{e:last.sum} corresponding to $d$  in the range \eqref{e:small.d}
is also $o(1/n)$, proving \eqref{eq;btm1.4} and, hence, completing the
proof of Proposition \ref{prop;ELT(2)}. 
\end{proof}

\begin{proof}[Proof of Theorem \ref{thm;extremal}]
We need to prove that for any fixed $0<T_1<T_2<\infty$, 
$$
    \left(\frac{\max_{s\leq nt}X_s-b_n}{a_n}, \, T_1\leqslant
      t\leqslant T_2\right)\Rightarrow
    \bigl( X_{\rm G}(t^{1-\beta}),\, T_1\leqslant t\leqslant T_2\bigr) 
$$
weakly in the Skorohod $J_1$ topology on $D[T_1,T_2]$, and without
loss of generality we assume that $T_2\leq 1$. 
According to \eqref{e:supm.to.extremal} and Proposition \ref{prop;X2n}, 
 is the same as proving
\begin{equation} \label{eq;main2}
    \left(\frac{ M_n^{(1)}([0,t])-b_n}{a_n}, \, T_1\leqslant
      t\leqslant T_2\right)\Rightarrow
    \bigl( \mathcal{M}([0,t]),  \, T_1\leqslant t\leqslant T_2\bigr) 
  \end{equation}
 in the same space. We construct all the random objects in
 \eqref{eq;main1} on the same probability space as in the proof of
 Theorem \ref{thm;main} and prove a.s convergence in $D[T_1,T_2]$. In
 the course of the proof of the latter theorem we have shown that for
 every $\vep>0$ there is $K\geq 1$ such that
 $$
 \limsup_{n\to\infty} \PP\Bigl[\bigl( M_n^{(1)}([0,t]), \, T_1\leqslant
      t\leqslant T_2\bigr)\not=\bigl( M_{n,K}([0,t]), \, T_1\leqslant
      t\leqslant T_2\bigr)\Bigr]\leq\vep\,.
$$
Since, clearly,
$$
\lim_{K\to\infty}\PP\Bigl[\bigl( \mathcal{M}_K([0,t]),  \,
T_1\leqslant t\leqslant T_2\bigr)  
=\bigl( \mathcal{M}([0,t]),  \, T_1\leqslant t\leqslant T_2\bigr)
\Bigr]=1\,.
$$
\eqref{eq;main2} will follow once we prove that for every $K=1,2,\ldots$
$$
\left(\frac{ M_{n,K}([0,t])-b_n}{a_n}, \, T_1\leqslant
      t\leqslant T_2\right)\rightarrow
    \bigl( \mathcal{M}_K([0,t]),  \, T_1\leqslant t\leqslant T_2\bigr) 
$$
a.s. in $D[T_1,T_2]$ as $n\to\infty$. The stochastic process in the
right hand side  
 may take the value $-\infty$; the probability
of this converges to zero as $K\to\infty$. For nondecreasing functions
the value of $-\infty$ introduces no difficulties in 
the $J_1$ topology if one interpretes $(-\infty)-(-\infty)$ as zero. 
The assumption $0<\beta<1/2$
implies that the stable regenerative sets $(\overline R_j)$ are a.s. disjoint, so
the latter statement will follows from 
$$
\left(\frac{ M_{n,(k)}([0,t])-b_n}{a_n}, \, T_1\leqslant
      t\leqslant T_2\right)\rightarrow
    \bigl( \mathcal{M}_{(k)}([0,t]),  \, T_1\leqslant t\leqslant T_2\bigr) 
$$
a.s. in $D[T_1,T_2]$ for every $k\geqslant 1$ and, as before, it is
enough to consider the case $k=1$. As in the proof of Proposition
\ref{prop:new.k}, we only need to check that
\begin{align} \label{e:e:main3}
&\left(
\frac{ V\left( w_n / \Gamma_1  \right) 
     \one_{\{ I_{1,n} \cap [0,nt]\neq \emptyset  \} }      -  V(w_n) }
     {h\circ V(w_n)}, \, T_1\leqslant   t\leqslant T_2\right) \\
 \rightarrow 
 & \bigl(   \mathcal{M}_{(1)}([0,t]) , \, T_1\leqslant   t\leqslant T_2\bigr)
 \notag 
 \end{align}    
 a.s.. However, it follows from (\ref{eq;main1.5}) that, a.s., 
 $$
 \inf\{I_{1,n}/n\}\to \inf\{\overline R_1\}\,.
 $$
Together with \eqref{e:first.t} this establishes \eqref{e:e:main3}, as
required. 
\end{proof}


\appendix 

\section{Random Walks with Regularly Varying Tails} \label{app;+rwrv}

Among the major goals of this appendix is to prove Theorem
\ref{thm;istP}. We start with recalling certain 
results on the ranges of the random walks from
\cite{barlow:taylor:1992} and   \cite{samorodnitsky:wang:2019}. We 
consider a  random walk $\{S_n\}_{n \geqslant 0}$ with $N_0$-valued
steps $\{\xi _n\}_{n \geqslant 1}$ whose distribution $F$ satisfies Assumption
\ref{ass;YRV0}. Recall the standard notions 
\begin{enumerate}[label=(\alph*)]
    \item the \textit{range} $A = \{S_n : n =0,1,2,\ldots   \}$, 
    \item the \textit{sojourn time} in $F$ up to time $k$,  $T_F(k)  =
      \# \{  0 \leqslant n \leqslant k : S_n \in F  \} $, for
      $F\subset \mathbb{N}_0, \, k \in \mathbb{N}\cup \{\infty\}$.  
\end{enumerate}
The following properties are well-known; see e.g. Appendix A in
\cite{samorodnitsky:wang:2019}.   As $n\to\infty$, 
\begin{align} 
    & \mathbb{E}_0 \, T_{[0,n]} (\infty) \sim \frac{n^ \beta }
    { \Gamma (1+\beta)  \Gamma(1-\beta) L(n) }\,, \label{eq;mtgA2}\\
    &  \mathbb{P}_0 \left(    A \cap \{n\}  \not=\emptyset \right ) \sim 
    \frac{n ^ {\beta-1}\mathbb{P}(\xi_1>0) }{ \Gamma (\beta)
      \Gamma(1-\beta) L(n)   }\,.
\label{eq;mtgA3}
\end{align}

For  $F\subset \mathbb{Z}$ we denote by $D(F):= \{ x-y: x, y \in F \}$
its  \textit{difference set}. For every $\delta \in (0,1)$, there
exists  $c_0  =c_0(\delta)>0$ such that for every $F$ and every $k \in
\mathbb{N}\cup \{\infty\}$ with   $0<\mathbb{E}_0 ( T_{D(F)} (k)
)<\infty$, we have 
\begin{equation} \label{eq;sjt1}
    \mathbb{P}_x
    \left( 
    T_F (k) \geqslant c \, \mathbb{E}_0 ( T_{D(F)} (k) )
    \right) \leqslant e^{-c \delta} \ \  \text{for each } \ 
    x \in F \ \text{and} \  c > c_0\,,
\end{equation}
see e.g.  Lemma 3.1 in \cite{pruitt:taylor:1969}. Choosing, in
particular, $F=F_n =\{0,1,\ldots,n\}$, we have by \eqref{eq;mtgA2}, 
$$
\mathbb{E}_0 T_{D(F_n)}(\infty) = \mathbb{E}_0 T_{F_n}(\infty) \lesssim \frac{n^\beta}{L(n)}.
$$
Therefor, choosing in \eqref{eq;sjt1} $\delta=1/2$, we see that for any
$C>0$ we  can choose $c>0$ so that for all $n \in
\mathbb{N}$,  
\begin{equation} \label{e:sjt2.1}
    \mathbb{P}_0 \left( \# (A \cap [0,n] ) \geqslant 
     \frac{c n^ \beta
       \log n }{ L(n)}   \right) \leqslant
   \mathbb{P}_0 \left(  T_{F_n}(\infty) \geqslant 
     \frac{c n^ \beta
       \log n }{ L(n)}   \right)   
 \leqslant n^{-C}. 
 \end{equation}
Furthermore, for a sufficiently large $c$, for each $n$, we have simultaneously for all $k=0,1,\ldots,n$ that 
  \begin{equation}  \label{eq;sjt2.2}
       \mathbb{P}_0 \left( \# (A \cap [0,2^k ) ) \geqslant 
     \frac{c n 2^ {\beta k}
     }{ L(2^k)}   \right)  \leqslant e^{-n}\,.
   \end{equation}
   
\begin{lemma}  \label{thm;range2}
Assume that  Assumption \ref{ass;YRV0} holds and $S_0=0$.  Then 
\begin{equation}  \label{eq;range2}
   \limsup_{n_0 \to \infty} \; \sup_{n>n_0}\; \max_{0 \leqslant k \leqslant n-1}
  \;  \max_{m\in \bbz}
  \frac{\#  ( A \cap [m,m+2^k) \cap [2^{n_0} , 2^n  ) ) }{ n 2^{\beta k} / L(2^k) }
  < \infty\quad \text{a.s.}    
\end{equation}
\end{lemma}
\begin{proof}

Let $c$ be such that \eqref{eq;sjt2.2} holds. Then 
  \begin{align*}
&\PP\left( \sup_{n>n_0}\max_{0 \leqslant k \leqslant n-1} \max_{m\in \bbz}
  \frac{\#  ( A \cap [m,m+2^k) \cap [2^{n_0} , 2^n  ) ) }{ n 2^{\beta k} /
                   L(2^k) }\geq c\right) \\
  \leq & \sum_{n>n_0}   n 2^n \cdot   \max_{
 \begin{subarray}{c}
0 \leqslant k \leqslant n-1 \\
[m,  m+2^k) \subset [2^{n_0}, 2^n ) 
\end{subarray}
 }  \,
\PP\left( \#(A \cap [m,m+2^k) \cap [2^{n_0} , 2^n  ))\geq
                  \frac{cn2^{\beta k}}{ L(2^k) }\right)    \\
 \leq & \sum_{n>n_0} n 2^{n}  \, \max_{0 \leqslant k \leqslant n-1}  \PP\left(\# ( A \cap [0,2^k))  \geqslant \frac{c
                  n 2 ^{\beta k}}{L(2^k)} \right) \leq \sum_{n>n_0} n (2/e)^n\,,
  \end{align*}
  where the 1st inequality follows by the union bound since 
  $$n 2^n \geqslant \# \{  [m,m+2^k) : [m,m+2^k) \subset [2^{n_0},2^n) , k=0,\ldots, n-1.  \}, $$
  the 2nd inequality follows from the strong Markov
 property, and the last one follows from  \eqref{eq;sjt2.2}. Since the last expression is summable in $n_0$,
 (\ref{eq;range2}) follows by the first Borel-Cantelli Lemma. 

\end{proof}

\begin{lemma}  \label{lem;sparse1}
Assume that  Assumption \ref{ass;YRV0} holds and $S_0=0$. For any 
$\eta, \gamma > 0$, 
\begin{equation} \label{eq;sparse1}
    \#  \left \{k: S_k \leqslant \eta n, \,\xi_k \geqslant (\log n)^\gamma   \right \}  \gtrsim_P 
     \frac{n^\beta  L\left(  (\log n)^\gamma  \right) }{   ( \log n
       )^{\gamma \beta} L(n) }\,. 
\end{equation} 
\end{lemma}
\begin{proof}
Let 
$N_t = \max \{ k: S_k \leqslant t   \} + 1, \, t\geq
0$. Then for each $x>0$, as $m\to\infty$, 
$$
    \mathbb{P} \left(
    \overline{F}(m) N_{m}  \geqslant  x ^ {-\beta}
    \right ) \rightarrow   J_\beta (x), 
$$
where $J_\beta$ is an $\bbr_+$-supported strictly $\beta$-stable distribution;
see e.g. \cite{feller:1966} \Rom{11}.5 (5.6). Therefore, for any
$\epsilon>0$ we can choose $c>0$ so small that with 
$m_n= \lceil cn^\beta / L(n)  \rceil$ we have $\liminf_{n \to \infty}
P(B_n) > 1 - \epsilon $ for the events   $B_n=\{  N_{\eta n} \geqslant
m_n \} $, $n\geq 1$. Consider also the events 
\begin{align*}
D_n  = \left \{      
 \frac{1}{m_n}\sum_{k=1} ^ {m_n} 1_{\{ \xi_k > (\log n ) ^ \gamma  \}}
                 <  \frac{  \overline{F}\left(  (\log n) ^\gamma
                 \right) } { 2 } \right \}, \ n=1,2,\ldots\,.
\end{align*}
By Chebyshev's inequality, as $n\to\infty$, 
$$
\mathbb{P}(D_n) \lesssim  \frac{ m_n \cdot  \text{var}( 1_{\{ \xi_k > (\log n ) ^ \gamma  \}} )   } 
{ \left( m_n  \overline{F} \left(  (\log n) ^\gamma \right)  \right) ^ 2    }  \lesssim 
\frac{L(n)}{n^\beta} \cdot
\frac{ (\log n)^{\gamma \beta} }{L \left((\log n)^ \gamma  \right ) }
\to 0\,.
$$
Hence  $\liminf_{n\to \infty} \mathbb{P}(B_n \cap D_n^c) \geqslant 1-
\epsilon$. However, on  the event $B_n \cap D_n^c$
$$
 \#  \left \{k: S_k \leqslant \eta n,\,  \xi_k \geqslant (\log n)^\gamma   \right \}  \geqslant 
 \frac{c}{2} \frac{n^\beta  L\left(  (\log n)^\gamma  \right) }{   (
   \log n )^{\gamma \beta} L(n) }\,,  
$$
leading to the desired conclusion. 
\end{proof}

\noindent 
{\it Proof of Theorem \ref{thm;istP}}  $(\rom{1})$. 
We use the notation $T=(a,b)$ throughout the proof.  Let
$\{Y^{(1)}\}_{t \in \mathbb{Z}}$ and $\{Y^{(2)}\}_{t \in \mathbb{Z}}$ be 
i.i.d. Markov chains  on $\mathbb{Z}$ starting at $0$, 
satisfying  Assumption \ref{ass;YRV0}. The simultaneous visit times of
the two chains to $0$,
$$
\varphi_j ^ \ast = \inf \{ n \geqslant \varphi_{j-1}+1 : Y^{(1)}_n =
Y^{(2)}_n = 0 \}, \ j=1,2,\ldots  
$$
with $\varphi_0^ \ast=0$ and $\varphi_j^ \ast=\infty$ if
$\varphi_{j-1}^ \ast=\infty$, form a terminating (since $\beta<1/2$)
renewal process. We denote by $\overline{F^\ast}$ the tail
distribution of $\varphi^\ast:= \varphi^\ast_1$. 
By the last entrance decomposition,  
\begin{align*}
  p_{n,T} &= \sum_{k=\lceil na \rceil} ^ { \lfloor nb \rfloor }
            \mathbb{P}\left( Y^{(1)}_k=
Y^{(2)}_k = 0,  \ \text{no simultaneous returns after $k$ in
            $nT$}\right) \\
  &= \sum_{k=\lceil na \rceil} ^ { \lfloor nb \rfloor } \frac{1}{w_n^2}  \overline{F^\ast}(\lfloor nb
    \rfloor -k)  \asymp \frac{n}{w_n^2} 
\end{align*}
since $\overline{F^\ast}(\infty)>0$. This proves (\ref{eq;anl1}).

\medskip
 
In the remaining part of the proof we will use the following simple
observation that allows us to use the basic facts about random walks
with regularly varying tails describing earlier in this section. For a
fixed $n \in  \mathbb{N}$ we construct a random walk
$\{S_k^{(n)}\}_{k\geqslant 0}$ by choosing the initial state
distributed as $\min I_{1,n}$ and the steps with the distribution
${F}$ in (\ref{eq;varphi1}). Recall that 
\begin{equation} \label{e:init.n}
\mathbb{P}(\min I_{1,n}\leq nx)=\frac{w_{[nx]}}{w_n} \ \ \text{for
  $0<x<1$.}
\end{equation}
The range $A_n$ of  $\{ S_k^{(n)}\}_{k\geqslant 0}$, obviously,  satisfies
$$
     A_n \cap [0,m] \stackrel{d}{=}  I_{1,n} \cap [0,m] 
     \text{  for all } m\leq n\,.
 $$
By conditioning on $S_0^{(n)}$ and using \eqref{e:sjt2.1}, we see that
for any $C>0$ we can choose $c>0$ so that for all $n \in  \mathbb{N}$, 
  \begin{equation}  \label{eq;obsB1.3}
     \mathbb{P}\left(
     \# I_{1,n} \geqslant \frac{cn^\beta \log n}{L(n)}
     \right) \leqslant n^{-C}\,.
 \end{equation}

\noindent 
{\it Proof of Theorem \ref{thm;istP}}  $(\rom{2})$. 
For $c>0$ denote $a_{c,n}  := c n^\beta \log n / L(n)$ and consider  the  event 
$
B_n =  \left \{ \#( I_{1,n} \cap n T ) \leqslant 
 a_{c,n} \right  \}$.
On $B_n$, the ``union bound method'' shows that for large $n$, 
$$
 \overline{p}_{n,T} \leqslant \frac{\# (I_{1,n} \cap nT) }{w_n }
 \leqslant  \frac{a_{c,n} }{w_n}.
$$
Therefore it suffices to show that $\mathbb{P}(B_n^c) \leqslant n^{-C}$ if
$c$ is large enough. This, however, follows immediately from
(\ref{eq;obsB1.3}). 

\medskip 
\noindent 
{\it Proof of Theorem \ref{thm;istP}}  $(\rom{3})$.
By the measure preserving property of the shift $\theta$ it is enough
to consider intervals of the form $T=(0,b)$. Let $v_1< v_2<v_3<
\cdots$ be the enumeration of the points of $I_{1,n}$ in the increasing
order. We construct a subset of by 
$I_{1,n}$  
\begin{equation} \label{eq;qch2.3}
   I_{1,n,\gamma} : = \{ v_i \in I_{1,n} :  v_{i+1} - v_{i} \geqslant (\log n) ^ \gamma  \}
 \end{equation}
(not including the last point in $I_{1,n}$.)
For an $\omega_1 \in \{ I_{1,n} \cap nT  \neq \emptyset \}$,  a lower
bound for $\overline{p}_{n,T}(\omega_1)$ is derived  below, where for
typographical convenience we use the notation $\mathbb{P}_2$ to denote
the probability measure associated with $Y^{(2,n)}$. 
\begin{align*}
    \overline{p}_{n,T} (\omega_1) 
    & \geqslant \mathbb{P}_2  \left(   
    I_{1,n,\gamma} (\omega_1) \cap I_{2,n} \cap n T \neq \emptyset  
     \right ) \\ 
     & \geqslant \sum_{u \in I_{1,n,\gamma}(\omega_1)\, \cap \, nT }
     \mathbb{P}_2 \left( u = \max ( I_{1,n,\gamma}(\omega_1) \cap  
     I_{2,n} ) \right) \\
     & = \frac{1}{w_n} \sum_{u \in I_{1,n,\gamma}(\omega_1)\, \cap \, nT}
      \mathbb{P}_2\left(   I_{1,n,\gamma}(\omega_1) \cap I_{2,n}  
      \cap (u, \infty) \cap nT  = \emptyset
      \, \big \vert \, u \in I_{2,n}  \right)\,.
\end{align*} 
Now the claim of part (iii) of the theorem follows from the
following two statements.

For any $\epsilon \in (0,1)$, there exists $c=c(\epsilon)>0$ such that
for all $\gamma>0$ 
\begin{equation}   \label{eq;qch2.4}
 \liminf_{n\to \infty} \mathbb{P} \left(  \#(I_{1,n,\gamma} \, \cap \, nT ) \geqslant
    d_n
     \, \big \vert \, 
     I_{1,n} \cap nT \neq \emptyset
     \right ) \geqslant 1 - \epsilon\,,
\end{equation}
where 
$$
d_n  =  \frac{c n^\beta  L\left(  (\log n)^\gamma  \right) }{   ( \log n )^{\gamma \beta} L(n) }\,.
$$
Further we claim that, if $\gamma > (1-2 \beta)^{-1}$, then for every
$0<\epsilon<1$ there is an event $B$ with $\mathbb{P}(B)>1-\epsilon$ such that
for every $w_1\in B$, 
\begin{equation} \label{eq;qch2.5}
\sup_{u \in I_{1,n,\gamma} ( \omega_1 ) \cap
nT}  \mathbb{P}_2\left(   I_{1,n,\gamma}(\omega_1) \cap I_{2,n}  
      \cap (u, \infty) \cap nT  \not= \emptyset
      \, \big \vert \, u \in I_{2,n}  \right)    = o_P(1)\,.
\end{equation} 
These two statements are proved in the remainder of this
section. 

Fix $\epsilon \in (0,1)$. 
For any $\eta \in (0,1)$ we have by \eqref{eq;mtg3.1} and
\eqref{eq;distV}, 
\begin{align*} 
& \mathbb{P} \left( \frac{1}{n} I_{1,n} \cap \eta n T \neq \emptyset   \, \Big \vert \,  \frac{1}{n} I_{1,n} \cap  n T \neq \emptyset \right)   \\
\to  & 
\mathbb{P} 
\left( \overline{R_1} \cap \eta T \neq \emptyset   \, \Big \vert \,  \overline{R_1}  \cap  T \neq \emptyset   \right ) = \eta^{1-\beta}.
\end{align*}
Therefore, if $\eta$ is sufficiently close to 1, 
\begin{equation}  \label{eq;qch3.1}
\liminf_{ n \to \infty  } \, \mathbb{P} \left( \frac{1}{n} I_{1,n} \cap \eta n T \neq \emptyset   \, \Big \vert \,  \frac{1}{n} I_{1,n} \cap  n T \neq \emptyset \right)   \geqslant \sqrt{1-\epsilon} . 
\end{equation}
Note that 
 \begin{align*}
 & \mathbb{P} \left( 
    \#(I_{1,n,\gamma} \, \cap \, nT ) \geqslant
    d_n
 \, \Big \vert \, 
      \frac{1}{n} I_{1,n} \cap \eta n T \neq \emptyset 
\right)  \\
= &  
\sum_{i \in \eta n T } P\bigl( S_0=i\big| S_0 \in \eta n  T\bigr)\, 
\mathbb{P}_i \left(  
\# \{ k: \xi_k \geqslant (\log n) ^ \gamma , S_k \leqslant nb- i  \}   \geqslant d_n
\right )   \\
\geqslant &  \mathbb{P}_0 \left( 
\# \{ k: \xi_k \geqslant (\log n) ^ \gamma , S_k \leqslant \lfloor n(1-\eta)b \rfloor  \}   \geqslant d_n\right)\,.
\end{align*} 

We conclude by Lemma \ref{lem;sparse1} that a fixed $\eta\in(0,1)$ for
which \eqref{eq;qch3.1} holds, we can choose  $c$ such that 
\begin{equation} \label{eq;qch3.2}
    \liminf_{n\to \infty} \mathbb{P} \left( 
     \#(I_{1,n,\gamma} \, \cap \, nT ) \geqslant
    d_n
     \, \big \vert \, 
       I_{1,n} \cap \eta n T \neq \emptyset
    \right) \geqslant \sqrt{1-\epsilon}\,.
\end{equation}
Clearly,  (\ref{eq;qch3.1}) and 
(\ref{eq;qch3.2}) give us  (\ref{eq;qch2.4}), so it remains to prove
\eqref{eq;qch2.5}.  

Let $k_1$ and  $k_2$ be such that 
$ 2^{k_1} \leqslant (\log n ) ^ \gamma < 2^{k_1 +1}$ and $2^{k_2 - 1 }
\leqslant n < 2^{k_2}$.   
Let $u \in I_{1,n, \gamma}(\omega_1)\cap nT $ and denote by 
$\overline{q}_n (u \vert \omega_1)$ the probability in the left hand
side of \eqref{eq;qch2.5}. We have 
\begin{align*}
    &\overline{q}_n (u \vert \omega_1)  \\
     \leqslant &  \sum_{k=k_1} ^ {k_2}  
    \mathbb{P}  \left( 
    s \in I_{2,n} \text{ for some } s \in [u+2^k, u+2^{k+1}) \cap 
    I_{1,n}(\omega_1)  \, \big \vert \, u \in I_{2,n}
    \right )  \\
     \leqslant & \sum_{k=k_1} ^ {k_2}  
    \# \left(  [u+2^k, u+2^{k+1}) \cap 
    I_{1,n}(\omega_1) \right )  \cdot 
    \max_{i\in [u+2^k, u+2^{k+1}) } \mathbb{P} \left(  i \in I_{2,n} \, \vert \, u \in I_{2,n}  \right)  . 
\end{align*}

Fix any $\epsilon\in (0,1)$. By \eqref{e:init.n} and 
Lemma \ref{thm;range2}, there is $C>0$ and an event $B$ with probability
higher than $1-\epsilon$  such that for all $n$ large enough, all $
w_1 \in  B$, all $u \in I_{1,n, \gamma}(\omega_1)\cap nT $ and all
$k\geq k_1$, 
\begin{equation} \label{eq;qch4.1}
     \# \left(  [u+2^k, u+2^{k+1}) \cap 
    I_{1,n}(\omega_1) \right ) \leqslant C
    \log n \frac{ 2^{\beta k}  }{L(2^k) }\,.
\end{equation}
Further, by \eqref{eq;mtgA3}, 
\begin{equation}  \label{eq;qch4.2}
 \sup_{u\geq 0}\, \max_{i\in [u+2^k, u+2^{k+1}) } \mathbb{P} \left(  i \in I_{2,n} \, \vert \, u \in I_{2,n}  \right)  
  \lesssim \frac{ 2^{-(1-\beta) k} } { L(2^k) }\,.
\end{equation}
Combining (\ref{eq;qch4.1}), (\ref{eq;qch4.2}), and Potter's bounds,
we see that for any $w_1\in B$ 
\begin{align*}
   \max_{u \in I_{1,n, \gamma}(\omega_1)\cap nT}  \overline{q}_n(u \vert \omega_1) \lesssim 
    \log n  \sum_{k=k_1} ^ {k_2}  \frac{2^{-(1-2\beta)k}}{( L(2^k) )^2}
  \lesssim   (\log n ) ^ { 1+\alpha \gamma } 
\end{align*}
for any $\alpha >2\beta-1$. By the choice of $\gamma$, we can select
$\alpha$ in such a way that $ 1+\alpha \gamma<0$. This proves
\eqref{eq;qch2.5}.


\section{Calculations for Sections \ref{sec;sidp} and
  \ref{sec;elt}}   \label{app;supp} 
We start by checking that the lognormal-type tails of Example
\ref{eg;lognm1}   satisfy  Assumption \ref{ass;nu}. The fact that
$(\nu(1,\infty))^{-1}\nu(\cdot\cap (1,\infty))$ is a subexponential
distribution follows from Theorem 4.1.17 in
\cite{samorodnitsky:2016}. Next, let 
$$\overline{H_{\#}}(x) = c_1 x^\beta (\log x ) ^ \xi \exp \left(  - 
\lambda (\log x ) ^ \gamma \right )$$
for $x>x_0$ that is large enough so that this function is decreasing
and $c_1$ is such that $\overline{H_{\#}}(x_0) =1$. That is,
\eqref{eq;f2} holds with 
$$
    h (x) =  
\left( \frac{\lambda \gamma (\log x )^{\gamma - 1}}{x} 
 - \frac{\xi}{x \log x}  -  \frac{\beta}{x}  \right ) ^ {-1} \,. 
$$
Regular variation of $h$ is clear, and so are the eventual positivity of
$h^\prime$ and the fact that $\lim_{x \to \infty} h^\prime (x) =0$. In
particular, $H_{\#} $ is in the maximum domain
  of  attraction of the Gumbel distribution.  Next, by the implicit
  function theorem, $G$ is, for large values of the argument, of the
  form \eqref{e:rep.G}. The relation $\overline{H_{\#}}\circ G (x) =
  x^{-1} $ for $x>1/\overline{H_{\#}} (x_0):=x_1$ means, in this case, that 
 $$
    c_1 G(x) ^ \beta ( \log G(x) ) ^ \xi \exp \left( -\lambda (\log G(x)) ^
      \gamma \right ) = x^{-1}\,, 
$$
so $\log G(x) \sim (\log x / \lambda ) ^ { 1/\gamma }$ as
$x\to\infty$. Denoting $g(x)=G^\prime(x)$ we also have, for $x>x_1$, 
$$
\beta \frac{g(x)}{G(x)} + \xi \frac{g(x) }{G(x) \log G(x)}  - \lambda
\gamma \frac{(\log G(x))^{\gamma-1} g(x) }{G(x)}  = -  \frac{1}{x} \,,
$$
so that, as $x\to\infty$, 
\begin{equation} \label{eq;lognm1.5}
   \zeta (x) =\frac{g(x)}{G(x)}x\log x 
\sim \gamma ^{-1}  \lambda ^ { -  1/\gamma } (\log x)^{ 1/\gamma } \,.
\end{equation}
The assumptions $(B1)$-$(B4)$ follow from (\ref{eq;lognm1.5}).

Next we check that the super-lognormal-type tails of Example
\ref{eg;suplognm1} satisfy  Assumption \ref{ass;nu}. Once again, the
fact that
$(\nu(1,\infty))^{-1}\nu(\cdot\cap (1,\infty))$ is a subexponential
distribution follows from Theorem 4.1.17 in
\cite{samorodnitsky:2016}. Now we set
$$\overline{H_{\#}}(x) = c_1 x^\beta (\log x ) ^ \xi \exp \left(  
\lambda (\log x ) ^ \gamma \right )
\exp \left(   - \rho \exp \left(  
\mu (\log x ) ^ \alpha \right ) \right )  $$
for or $x>x_0$ and appropriate $x_0,c_1$, and \eqref{eq;f2} holds with 
 $$
 h(x) =  
\left(  \frac{\rho \alpha \mu (\log x)^{\alpha-1} \exp \left(     \mu (\log x) ^ \alpha \right ) }{x}  -  
\frac{\lambda \gamma (\log x) ^{\gamma - 1}}{x}  - \frac{\xi}{x \log x}  - \frac{\beta}{x} \right ) ^ {-1} .    
$$
All of the arguments we used in the previous example still work. In
this case we have
$$
\exp \left(  
\mu (\log G(x) ) ^ \alpha \right ) \sim \log x/\rho \ \ \text{as
$x\to\infty$,}
$$
so also 
$\log G(x)  \sim  (  \log \log x/\mu ) ^ {1/\alpha}$
as $x\to\infty$. Since 
\begin{align*}
\beta \frac{g(x)}{G(x)} +& \xi \frac{g(x) }{G(x) \log G(x)}  + \lambda
  \gamma \frac{(\log G(x))^{\gamma-1} g(x) }{G(x)} \\
 -&\rho\mu\alpha \exp \left(  
\mu (\log x ) ^ \alpha \right ) \frac{(\log G(x))^{\alpha-1} g(x) }{G(x)}
= -  \frac{1}{x} \,,
\end{align*}
we conclude that     
\begin{equation}
    \zeta(x) =\frac{g(x)}{G(x)}x\log x 
\sim \alpha^{-1} \mu ^ { - 1/\alpha } \left( 
\log \log x 
\right ) ^ {(1-\alpha)/\alpha } \ \ \text{as $x\to\infty$.}
    \label{eq;suplognm1.4}      
  \end{equation}
  As before, the assumptions $(B1)$-$(B4)$ follow from
  (\ref{eq;suplognm1.4}). 

 \begin{proof}[Proof of Proposition \ref{prop;MTG4}]
\phantom{blank}

$(\rom{1})$ and $(\rom{2})$ follow  by direct integration. 
To show $(\rom{3})$, we note that the derivative $g$ of $G$ satisfies
$h\circ G(x)=xg(x)$ for all large $x$. Therefore, for large $x$, 
\begin{align*}
    & \frac{h\circ G(H_1 (x) )}
{h\circ G(H_2 (x)) }   
=  \frac{G(H_1 (x) )}{G(H_2 (x))} \cdot 
\frac{ \zeta(H_1 (x)) }
{\zeta(H_2 (x))  } 
\cdot 
\frac{ \log ( H_2(x) ) }
{\log ( H_1(x) )}    \\
\asymp  &  \frac{G(H_1 (x) )}{G(H_2 (x))}  \cdot \frac{ \zeta(H_1 (x)) }
{\zeta(H_2 (x))  } \gg  \exp \left\{ b ( \log \log x  ) ^ \delta  
\right \}
\end{align*}
by Assumption \ref{ass;nu} $(B2)$, Potter's bounds  and direct integration. 

For part $(\rom{4})$, we only consider the case 
$\alpha>0$. When $\alpha <0$, a similar argument works. Write 
$$
    G \left (x  (\log x) ^ \alpha  \right ) - G(x) 
   =
   \int_1 ^ {( \log x ) ^ \alpha } \frac{G(ux) \zeta( ux )  }{u \log(ux)} du \,.
$$
Dividing this identity by $h \circ G(x) = G(x) \zeta(x) / \log x $
gives  us 
\begin{equation}  \label{eq;C2.1}
    \frac{    G \left (x  (\log x) ^ \alpha  \right ) - G(x)  }
    { h \circ G(x)  }  
    =    \int_1 ^ {( \log x ) ^ \alpha} \frac{G(ux)}{G(x)} \cdot 
    \frac{\zeta(ux)}{\zeta(x)} \cdot 
    \frac{\log x}{\log (ux)}  \cdot
    \frac{du}{u} \,.
\end{equation}
Denote $I=[1 , (\log x)^\alpha ]$. Clearly, $ \log x \sim \log (ux)$ 
uniformly over $u\in I$. Further,  by Assumption \ref{ass;nu} $(B1),
(B3)$, we see that   $\zeta(x) \asymp \zeta (u x) $ uniformly over $u\in I$. 
Finally, for $u\in I$, by Assumption \ref{ass;nu} $(B2)$, 
\begin{align*}
1&\leq    \frac{G(ux)}{G(x)} = \exp\left\{\int_x ^ {ux }
   \frac{\zeta(v)}{v \log v} dv \right\}  \\
  &\leq \exp\left\{C\int_x ^ {ux }\frac{1}{v \log\log v} dv \right\}
  \\
  &\leq \exp\left\{C\int_x ^ {x(\log x)^\alpha }\frac{1}{v \log\log v} dv \right\}
    \to e^{\alpha C}\,,
\end{align*}
where $C$ is a suitable constant. The claim now follows from \eqref{eq;C2.1}
since 
 $$
  \int_1 ^ {( \log x ) ^ \alpha} \frac{du}{u} =\alpha \log \log x \,.
$$

The argument for $(\rom{5})$ is similar to that for $(\rom{4})$,  We
start with 
\begin{equation} \label{eq;C2.5}
  \frac{G(x) - G\left( x 2^{- j} \right )}
 { h\circ G(x) } 
 =  \int_{2^{- j}} ^ 1 
 \frac{G(x u)}{G(x)} \cdot \frac{\zeta(x u)}{\zeta(x)} 
 \cdot 
 \frac{\log x}{\log (x u )} \cdot \frac{du}{u} . 
\end{equation}
Denoting now 
$I =\left [2^{- \rho \log x / \zeta (x) } , 1   \right ]$. Due to $\zeta (\cdot) \to \infty$, it is clear  that $   \log x  \sim \log
(xu) $ 
uniformly over $u \in I$.  Furthermore,  $ x 2^{-\rho \log x / \zeta
  (x) } \to \infty$, so  by Assumption  \ref{ass;nu}  $(B1)$, $(B2)$ and
$(B3)$, 
$$
  \frac{\zeta(xu)}{\zeta(x) } \gtrsim 
\frac{\zeta\left( x 2^{-\rho \log x / \zeta (x) } \right )}
{\zeta(x) }  \gtrsim 
\frac{\zeta\left( x 2^{-\rho \log x / (\log \log x )^\delta } \right )}
{\zeta(x) }  \gtrsim 
1\,, 
$$
uniformly over $u \in I$. Finally, for $u\in I$, by Assumption
\ref{ass;nu} $(B1)$, $(B2)$,  for some constant $C$,
\begin{align*}
\frac{G(ux)}{G(x)}&\gtrsim \frac{G\left(  x 2^{- \rho \log x / \zeta (x) } \right )}{G(x)} 
=  \exp \left ( - \int_{ x 2^{-\rho \log x / \zeta (x) }} ^ x
                    \frac{\zeta(u)}{u \log u } d u  \right ) \\
&\geqslant   \exp \left( - 
 C \zeta(x)  \int_{ x 2^{-\rho \log x / \zeta (x) }} ^ x  \frac{du}{u \log u }  
\right ) > 2 ^ {\rho C - 1}
\end{align*}
for all $x$ large, uniformly over $u \in I$. Therefore, by 
   (\ref{eq;C2.5}) and  Assumption
\ref{ass;nu} $(B2)$, 
$$
\frac{G(x) - G\left( x 2^{- j} \right )}
 {j h \circ G(x) }  \gtrsim  1\,,
$$
as required. 
\end{proof}

We finish by checking the claims made in Remark \ref{rk:key}, and we
start with the  lognormal-type tails of Example \ref{eg;lognm1}. To
see that  \eqref{e:no.second} holds, it is enough to check that for
any $C>0$,
\begin{equation} \label{e:no.sec}
  \lim_{n\to\infty}\frac{G\bigl( 1/\overline{F} (n)\bigr)}{h\circ
    G(cw_n)}=0\,. 
\end{equation}
The ratio above is asymptotic to 
\begin{equation} \label{e:no.sec1}
\frac{G\bigl( 1/\overline{F} (n)\bigr)}{ G(cw_n)} \frac {\log w_n}{ \zeta(Cw_n)}
 =  \exp \left (
 \int^{1/\bar F(n)} _ {w_n} \frac{\zeta(u)}{u \log u} du 
 \right )   \frac {\log w_n}{ \zeta(Cw_n)}\,,
\end{equation}
which converges to 0 as $n\to\infty$ by \eqref{eq;lognm1.5}. Next, for 
the super-lognormal-type tails of Example \ref{eg;suplognm1} with
$0<\alpha<1/2$ one checks that  \eqref{e:no.second} holds in the same
way as above, by using \eqref{eq;suplognm1.4} instead of
\eqref{eq;lognm1.5}.  Finally, to see that \eqref{e:no.second} fails
when $1/2<\alpha<1$, one needs to prove that, in this case, for any
$C>0$ the limit in \eqref{e:no.sec} is infinity instead of 0. To do so
one uses, once again, \eqref{e:no.sec1}. It is
routine to see that the expression there converges to infinity by
using  \eqref{eq;suplognm1.4}.


\bibliographystyle{authyear}
\bibliography{bibfile}

\end{document}